\documentclass[10pt]{article}
\usepackage[singlespacing]{setspace}
\usepackage{amsmath, amssymb, amscd, amsthm, amsfonts}
\usepackage{mathtools}
\usepackage[mathscr]{euscript}
\usepackage{mathrsfs}
\usepackage{enumitem}
\usepackage{graphicx}
\usepackage{float}
\usepackage{tocloft}
\usepackage{caption}
\usepackage{subcaption}
\RequirePackage{doi}
\usepackage{hyperref}
\hypersetup{
colorlinks = true,
urlcolor   = blue,
citecolor  = blue,
linkcolor  = blue,
}
\usepackage{xcolor}
\usepackage{scalerel,stackengine}
\usepackage[colorinlistoftodos]{todonotes}
\usepackage{authblk}
\usepackage{fancyhdr}

\newlength\FHoffset
\fancyheadoffset{\FHoffset}

\newlength\FHleft
\newlength\FHright

\renewcommand{\headrulewidth}{1.0pt} 
\newbox\FHline
\setbox\FHline=\hbox{\hsize=\paperwidth%
  \hspace*{\FHleft}%
  \rule{\dimexpr\headwidth-\FHleft-\FHright\relax}{\headrulewidth}\hspace*{\FHright}%
}

\title{\textbf{Nonuniqueness of trajectories on a set of full measure for Sobolev vector fields}}
\author[]{\large \textbf{Anuj Kumar}\footnote{
Current affiliation: Department of Mathematics, University of California Berkeley, CA 94720. 

\hspace{0.1cm} \textit{Email:}  \href{mailto:anujkumar@berkeley.edu}{anujkumar@berkeley.edu}.

\vspace{0.2cm}

\hspace{0.1cm} Previous affiliation: Department of Applied Mathematics, University of California Santa Cruz, CA 95064. 

\hspace{0.1cm} \textit{Email:}  \href{mailto:akumar43@ucsc.edu}{akumar43@ucsc.edu}. }}
\date{}

\newtheoremstyle{mystyle}
  {}
  {}
  {\itshape}
  {}
  {\bfseries}
  {.}
  { }
  {\thmname{#1}\thmnumber{ #2}\thmnote{ (#3)}}

\theoremstyle{mystyle}
\newtheorem{theorem}{Theorem}[section]
\newtheorem{proposition}[theorem]{Proposition}
\newtheorem{lemma}{Lemma}[section]

\newtheorem{corollary}[theorem]{Corollary}

\newtheorem{definition}{Definition}[section]
\newtheorem{question}[theorem]{Question}

\theoremstyle{definition}

\newcommand\norm[1]{\left\lVert#1\right\rVert}

\newcommand{\bs}[1]{\boldsymbol{#1}}

\newcommand{\wt}[1]{\widetilde{#1}}
\newcommand{\ol}[1]{\overline{#1}}

\stackMath
\newcommand\reallywidecheck[1]{%
\savestack{\tmpbox}{\stretchto{%
  \scaleto{%
    \scalerel*[\widthof{\ensuremath{#1}}]{\kern-.6pt\bigwedge\kern-.6pt}%
    {\rule[-\textheight/2]{1ex}{\textheight}}
  }{\textheight}%
}{0.5ex}}%
\stackon[1pt]{#1}{\scalebox{-1}{\tmpbox}}%
}

\DeclareMathOperator\supp{supp}

\def\XXint#1#2#3{{\setbox0=\hbox{$#1{#2#3}{\int}$ }
\vcenter{\hbox{$#2#3$ }}\kern-.6\wd0}}

\numberwithin{equation}{section}

\begin{document}

\maketitle

\begin{abstract}
In this paper, we resolve an important long-standing question of Alberti \cite{alberti2012generalized} that asks if there is a continuous vector field with bounded divergence and of class $W^{1, p}$ for some $p \geq 1$ such that the ODE with this vector field has nonunique trajectories on a set of initial conditions with positive Lebesgue measure? This question belongs to the realm of well-known DiPerna--Lions theory for Sobolev vector fields $W^{1, p}$. In this work, we design a divergence-free vector field in $W^{1, p}$ with $p < d$ such that the set of initial conditions for which trajectories are not unique is a set of full measure. The construction in this paper is quite explicit; we can write down the expression of the vector field at any point in time and space. Moreover, our vector field construction is novel. We build a vector field $\bs{u}$ and a corresponding flow map $X^{\bs{u}}$ such that after finite time $T > 0$, the flow map takes the whole domain $\mathbb{T}^d$ to a Cantor set $\mathcal{C}_\Phi$, i.e., $X^{\bs{u}}(T, \mathbb{T}^d) = \mathcal{C}_\Phi$ and  the Hausdorff dimension of this Cantor set is strictly less than $d$. The flow map $X^{\bs{u}}$ constructed as such is not a regular Lagrangian flow. The nonuniqueness of trajectories on a full measure set is then deduced from the existence of the regular Lagrangian flow in the DiPerna--Lions theory.
\end{abstract}

\section{Introduction}
\label{Intro}
In this paper, we consider the following system of ordinary differential equations (ODE)
\begin{eqnarray}
    \frac{d \bs{x}(t)}{dt} = \bs{u}(t, \bs{x}(t)) \qquad \text{with} \qquad \bs{x}(0) = \bs{x}_0 \in \mathbb{T}^d,
    \label{The ODE}
\end{eqnarray}
where $\mathbb{T}^d$ is a $d$-dimensional Torus and $\bs{u}:[0, T] \times \mathbb{T}^d \to \mathbb{R}^d$ is a vector field. We are interested in studying an important question regarding this ODE for the case when the vector field $\bs{u}$ has only Sobolev regularity. To describe the complete problem, let us start with a few definitions. 
\begin{definition}[Trajectory]
\label{def: Trajectory}
We say $\gamma^{\bs{u}}_{\bs{x}} : [0, T] \to \mathbb{T}^d$ is a trajectory corresponding to the vector field $\bs{u}$ starting at $\bs{x}$ if $\gamma^{\bs{u}}_{\bs{x}}$ is absolutely continuous, $\gamma^{\bs{u}}_{\bs{x}}$ solves the ODE (\ref{The ODE}), i.e., $\dot{\gamma\hspace{1pt}}^{\bs{u}}_{\bs{x}}(t) = \bs{u}(t, \gamma^{\bs{u}}_{\bs{x}}(t))$ $\forall t \in [0, T]$ and $\gamma^{\bs{u}}_{\bs{x}}(0) = \bs{x}$.
\end{definition}
\noindent
For a vector field $\bs{u}$, by bundling these trajectories for $\mathscr{L}^d$-a.e. $\bs{x} \in \mathbb{T}^d$, we can define a flow map as follows.
\begin{definition}[Flow map]
\label{def: Flow map}
    A map $X^{\bs{u}} : [0, T] \times \mathbb{T}^d \to \mathbb{T}^d$ is called a flow map corresponding to the vector field $\bs{u}$ if, for $\mathscr{L}^d$-a.e. $\bs{x} \in \mathbb{T}^d$, $X^{\bs{u}}(\cdot, \bs{x}) : [0, T] \to \mathbb{T}^d$ is a trajectory starting from $\bs{x}$. 
\end{definition}
With this definition, we say two flow maps $X^{\bs{u}}_1$ and $X^{\bs{u}}_2$ are the same if  trajectories $X^{\bs{u}}_1(\cdot, \bs{x})$  and $X^{\bs{u}}_2(\cdot, \bs{x})$ are the same for $\mathscr{L}^d$-a.e. $\bs{x} \in \mathbb{T}^d$.
A restricted class of flow maps is called \emph{regular Lagrangian flows}, which plays an important role in DiPerna--Lions theory \cite{DiPernaLions}. The following definition is taken from a paper by Ambrosio \cite{Ambrosio04}.
\begin{definition}[Regular Lagrangian flow]
\label{def: Regular Lagrangian flow}
A map $X^{\bs{u}}_{\textit{\tiny RL}} : [0, T] \times \mathbb{T}^d \to \mathbb{T}^d$ is a regular Lagrangian flow if it is a flow map corresponding to the vector field $\bs{u}$. In addition, it satisfies the following condition.\\
\textbullet \; For any time $t \in [0, T]$, $X^{\bs{u}}_{\textit{\tiny RL}}(t, \cdot)_{\#}\mathscr{L}^d \leq C \mathscr{L}^d$ for some constant $C > 0$.
\end{definition}
If the vector field is Lipschitz continuous, then the classical Cauchy--Lipschitz theorem guarantees the uniqueness of trajectories starting from any $\bs{x}$. Moreover, the corresponding flow map $X^{\bs{u}}$ is automatically a regular Lagrangian flow, i.e., the additional condition required in the definition (\ref{def: Regular Lagrangian flow}) is a consequence in the classical theory.

In the late 1980s, DiPerna and Lions \cite{DiPernaLions}, in a pioneering work, developed the theory of ODE for Sobolev vector fields with bounded divergence, where they showed the existence and uniqueness in the class of regular Lagrangian flow. The theory of DiPerna and Lions was later extended by Ambrosio \cite{Ambrosio04} to vector fields of bounded variation. Since the theory of DiPerna and Lions first came out, an interesting question remained: is it possible to show the uniqueness of a general flow map, as in definition (\ref{def: Flow map}), for Sobolev vector fields? In other words, is any flow map automatically a regular Lagrangian flow, just as in the classical case? In particular, there is a question by G. Alberti \cite{alberti2012generalized}, also stated in AMS 2023 Colloquium Lectures  by Prof. C. De Lellis \cite{delellisAMSlecture2023},  who asked:
\begin{question}[Alberti \cite{alberti2012generalized}]
Is there a continuous vector field with bounded divergence and of class $W^{1, p}$ for some $p \geq 1$ such that the nonuniqueness set has positive measure?
\label{Q: alberti}
\end{question}

The word `continuous' in above question is important as one falls under the Peano existence theorem if the vector field is continuous. In recent year, progress has been made to resolve this question. Caravenna and Crippa \cite{CaravennaCrippa18} showed that if $\bs{u} \in C([0, T]; W^{1, p}(\mathbb{T}^d, \mathbb{R}^d))$ with $p > d$, then the trajectories are unique for  $\mathscr{L}^d$-a.e. $\bs{x} \in \mathbb{T}^d$. 

If $p < d$, then using the method of convex integration Bru\'e, Colombo and De Lellis \cite{BrueColomboDeLellis21} constructed a divergence-free vector fields $\bs{u} \in C([0, T]; W^{1, p}(\mathbb{T}^d, \mathbb{R}^d) \cap L^s)$ for $s < \infty$ such that the trajectories are not unique on a set of positive measure, which also implies that there are flow maps $X^{\bs{u}}$ that are not regular Lagrangian. The method of convex integration in fluid dynamics was initially developed and introduced by De Lellis \& Sz\'ekelyhidi \cite{de2009euler, de2013dissipative}. The convex integration method was used to resolve the flexibility part of Onsager's conjecture \cite{isett2018proof}, which states that for every $\alpha < 1/3$ there exists a $C^{\alpha}_{t, \bs{x}}$ solution of the Euler equation which do not conserve energy. The reader can refer to the following studies \cite{buckmaster2015onsager, buckmaster2015anomalous, daneri2017non, buckmaster2019onsager, buckmaster2019nonuniqueness, buckmaster2021wild} for the application of the convex integration method in Euler and Navier--Stokes equations. 
Recently, in a groundbreaking work, Modena 
 \& Sz\'ekelyhidi \cite{modena2018non, modena2019non} brought the convex integration technique to the transport equation to construct nonunique solutions where the vector field is divergence-free and belongs to Sobolev space. Modena and Sattig \cite{modena2020convex} subsequently improve the range of integrability exponents. In both studies, the solution of the transport equation takes both positive and negative values.


Recently, \cite{BrueColomboDeLellis21}, constructed nonunique positive solutions of the transport equation with the same integrability range as \cite{modena2020convex}. \cite{BrueColomboDeLellis21} then proved the nonuniqueness of trajectories of the ODE system via Ambrosio's superposition principle which works for the positive solution of the transport equation. The same result regarding the nonuniqueness of trajectories cannot be concluded from the studies of \cite{modena2018non} and \cite{modena2020convex} because of the absence of a superposition principle for solutions of the transport equation that attain both positive and negative values. Using the same methodology, Giri and Sorella \cite{GiriSorella21} covered the case of autonomous vector fields. They constructed a divergence-free vector field $\bs{u} \in W^{1, p}(\mathbb{T}^d, \mathbb{R}^d)$ with $p < d-1$ and showed that the nonuniqueness of trajectories could occur on a set of positive measure. Finally, Pitcho and Sorella \cite{pitcho2023almost} using the convex integration showed that the nonuniqueness of trajectories can happen almost everywhere in space. The construction of \cite{BrueColomboDeLellis21, GiriSorella21, pitcho2023almost}, however, does not resolve Alberti's question as their vector field is not continuous in space. Besides the studies using the convex integration technique, we note that explicit examples of vector fields for which nonuniqueness happens on a set of full Hausdorff dimension but of measure zero are also known  \cite{FeffermanPooleyRodrigo21}.

Inspired by one of our recent studies on branching flow for optimal heat transport \cite{kumar2022three, kumar2023bulk}, in this paper, we give an explicit example of a vector field $\bs{u} \in C([0, T]; W^{1, p}(\mathbb{T}^d, \mathbb{R}^d)$ with $p < d$, for which we show that the set of initial conditions for which trajectories are not unique can be of full measure. Furthermore, the vector field that we construct also possesses H\"older continuity of exponent $\alpha$ arbitrarily close to one, which then resolves the Question \ref{Q: alberti} of Alberti.  The following theorem summarizes our result for the unsteady case.

\begin{theorem}[Main result: unsteady case]
\label{thm unsteady}
For every integer $d \geq 2$ and  every $1 \leq p < d$, $0 < \alpha < 1$ and $T > 0$, there exists a divergence-free vector field $\bs{u} \in C([0, T]; W^{1, p}(\mathbb{T}^d, \mathbb{R}^d)) \cap C^{\alpha}([0, T] \times \mathbb{T}^d, \mathbb{R}^d)$ such that the set of initial conditions for which trajectories are not unique is a full measure set.
\end{theorem}
The unsteady vector field construction to prove the above theorem can be trivially modified to produce a similar result using a steady vector field. Essentially, we trade time for one space dimension. The result is summarized in the following theorem.
\begin{theorem}[Main result: steady case]
\label{thm steady}
For every integer $d \geq 3$ and  every $1 \leq p < d-1$ and $0 < \alpha < 1$, there exists a divergence-free vector field $\bs{u}^s \in W^{1, p}(\mathbb{T}^d, \mathbb{R}^d) \cap C^{\alpha}(\mathbb{T}^d, \mathbb{R}^d)$ such that the set of initial conditions for which trajectories are not unique is a full measure set.
\end{theorem}

The proof of the theorem in the unsteady case case involves a Cantor set construction $\mathcal{C}_\Phi$ of measure zero. The vector field we designed to prove Theorem \ref{thm unsteady} is divergence-free and it translates (in space) cubes of various generations from the Cantor set $\mathcal{C}_\Phi$ construction. The translation of a typical cube is carried out using a `blob flow' previously introduced by the author in context of optimal heat transport \cite{anujkumarstonybrook22}. Corresponding to this vector field, we design a flow map $X^{\bs{u}}$ that takes the whole domain $\mathbb{T}^d$ and maps it to the Cantor set $\mathcal{C}_\Phi$. The construction carried out in this paper is quite explicit; we can write down the precise mathematical expression of the vector field and a trajectory from the flow map $X^{\bs{u}}$ for any initial condition as a function of time. Other than the H\"older continuity of our vector field, this aspect of our construction makes it different from the previous study of \cite{BrueColomboDeLellis21, GiriSorella21, pitcho2023almost} where even though the nonuniqueness of trajectories on a set positive measure was shown it is not clear how individual trajectories move in space as a function of time. Furthermore, the mechanism behind the vector field design is novel and has not previously been employed in any other study.

The importance of flow designs in fluid research can hardly be overstated. Flow designs are used to produce examples of nonuniqueness \cite{depauw2003non, de2022smoothing}, to answer questions related to anomalous dissipation (a typical characteristic of turbulent flow) in a passive scalar \cite{drivas22anomdissp, colombo2023anomalous, armstrong2023anomalous, elgindi2023norm} or velocity field \cite{brue2023anomalous}, as well as in problems related to mixing \cite{AlbertiCrippaMazzucato19, YaoZlatos17, ElgindiZlatosuniversalmixer}. In all of these studies, the two most popular flow mechanisms by far are the checkered board flow and alternating shear. However, both of these mechanisms do not appear to work to produce a nonuniqueness example for the Question \ref{Q: alberti} asked by Alberti \cite{alberti2012generalized}. Therefore, our work adds one more mechanism to the list of already-known mechanisms that we believe can potentially be very useful in problems related to nonuniqueness, anomalous dissipation, and mixing. Indeed, based on the initial idea from this paper, \cite{bruecolombokumar24} has constructed sophisticated vector fields for which there are two solutions to the transport equation, going beyond the integrability exponent range previously covered using the method of convex integration and matching sharply with the known uniqueness boundary \cite{BrueColomboDeLellis21}.


\subsection{Organization of the paper}
The paper is arranged as follows. As our design of vector field $\bs{u}$ will involve a Cantor set, we give the required Cantor set construction and prove a few associated basic lemmas in section \ref{sec: Cantor set}. We give an overview of the design of the vector field $\bs{u}$ in section \ref{section: overview}. In section \ref{sec: proof of thms}, we give proof of Theorem \ref{thm unsteady} and briefly summarize the changes required to prove Theorem \ref{thm steady}. Finally, in section \ref{sec: blob flow}, we construct a useful flow that helps translate cubes in the domain and is essential in the design of vector field $\bs{u}$.


\section{Cantor set construction}
\label{sec: Cantor set}
The purpose of this section is to fix some notations and give a Cantor set construction in a way that is readily usable throughout the paper. We also state and prove a few basic lemmas required later for constructing the vector field.

\noindent
\textbf{Notation 1.} 
In this paper, we will work with both $\mathbb{R}^d$ and a $d$-dimensional torus $\mathbb{T}^d = \mathbb{R}^d/\mathbb{Z}^d$. We identify point $\bs{x} \in \mathbb{R}^d$ or $\mathbb{T}^d$ through its components as $\bs{x} = (x_1, \dots, x_d)$, where $x_i \in \mathbb{R}$ or $\mathbb{T}$ for $1 \leq i \leq d$. We will also use the notation $(\bs{x})_i$ to denote the $i$th component of $\bs{x}$ when it is more convenient to do so. We use $\mathscr{L}^d$ to mean the $d$-dimensional Lebesgue measure on $\mathbb{R}^d$ or $\mathbb{T}^d$. Given a vector field $\bs{u}$, we define the support of $\bs{u}$ in the time variable as
\begin{eqnarray}
    \supp_t \bs{u} \coloneqq \ol{\{t \in \mathbb{R} | \bs{u}(t, \bs{x}) = \bs{0} \; \forall \, \bs{x} \in \mathbb{R}^d \text{ or }\mathbb{T}^d\}}.
\end{eqnarray}
Throughout the paper, we will use $a \lesssim b$ to mean that $a < c \, b$ for some constant $c > 0$ independent of any parameter, except we will allow this constant $c$ to depend on the dimension $d$.
Given a point $\bs{x}_c$ in $\mathbb{R}^d$ or $\mathbb{T}^d$ and $\ell > 0$, we define an open cube of length $\ell$ centered at $\bs{x}_c$ as follows
\begin{eqnarray}
    Q(\bs{x}_c, \ell) \coloneqq \left\{\bs{x} \in \mathbb{R}^d \text{ or } \mathbb{T}^d \, \left| \, |x_i - x^c_i| < \frac{\ell}{2}, \; 1 \leq i \leq d \right. \right\}, 
\end{eqnarray}
and a close cube $\overline{Q}(\bs{x}_c, \ell)$ to be the closure of $Q(\bs{x}_c, \ell)$.

\noindent
\textbf{Notation 2.}
In preparation for constructing Cantor sets, we first define a few important sets and sequences. We let,
\begin{eqnarray}
    \mathbb{I} \coloneqq \{1, -1\}, \qquad \text{and} \qquad  \mathbb{I}^d \coloneqq \{\bs{s} = (s^1, \dots, s^d) \, | \, s^i \in \mathbb{I}\}. 
\end{eqnarray}
For every $n \in \mathbb{N}$, we define a set of $n-$tuples with elements from $\mathbb{I}^d$ as
\begin{eqnarray}
    S_n \coloneqq \{\mathfrak{s}  = (\bs{s}_1, \dots, \bs{s}_n) \, | \, \bs{s}_i \in \mathbb{I}^d, 1 \leq i \leq n\}. 
\end{eqnarray}
It is clear that $|S_n| = 2^{nd}$. We define a set of sequences with elements from $\mathbb{I}^d$ as
\begin{eqnarray}
    \mathcal{S} \coloneqq \{\bs{\mathfrak{s}} = (\bs{s}_1, \bs{s}_2, \dots ) \, | \, \bs{s}_i \in \mathbb{I}^d, i \in \mathbb{N}\}. 
\end{eqnarray}
We note that we use Fraktur font to denote elements from a  set $S_n$ and \emph{bold} Fraktur font to denote elements from the set $\mathcal{S}$.
Given $\mathfrak{s} \in S_n$ with $n \geq 2$, we will denote
\begin{eqnarray}
    \mathfrak{s}^\prime \coloneqq (\bs{s}_1, \dots, \bs{s}_{n-1}) \in S_{n-1},
    \label{def: prime notation mathfrak s}
\end{eqnarray}
and we define $\sigma_n : \mathcal{S} \to S_n$ as follows
\begin{eqnarray}
    \sigma_n(\bs{\mathfrak{s}}) \coloneqq (\bs{s}_1, \dots \bs{s}_n) \qquad \text{for} \qquad \bs{\mathfrak{s}} \in \mathcal{S}. 
\end{eqnarray}

\noindent
\textbf{Notation 3.}
Consider a sequence $\Psi \coloneqq \{\psi_i\}_{i = 1}^\infty$ with elements $0 < \psi_i \leq 1$ and a sequence of $1$'s 
\begin{eqnarray}
    \Theta \coloneqq \{1, 1, \dots \}. 
    \label{def: Theta seq of 1s}
\end{eqnarray}
We define a few lengths corresponding to $\Psi$ as
\begin{eqnarray}
   \ell^0_\Psi = 1 \qquad \qquad \text{and} \qquad \qquad \ell^n_\Psi = \frac{\prod\limits_{i = 1}^n \psi_i}{2^n} \qquad \text{for} \qquad n \in \mathbb{N}. 
\end{eqnarray}
With this definition and the fact that $0 < \psi_i \leq 1$, we note that 
\begin{eqnarray}
    \ell^n_\Psi \leq \frac{\ell^{n-1}_\Psi}{2} \qquad \text{for} \qquad n \in \mathbb{N}. 
\end{eqnarray}
The significance of $\ell^n_\Psi$ is shown in figure \ref{fig: Cantor set}. The length $\ell^n_\Psi$ represents the size of the $n$-th generation cubes in the Cantor set construction process or the $n$-th generation dyadic cubes if $\Psi$ is chosen to be $\Theta$ (the sequence of $1$'s).

Next, for a given sequence $\Psi$ as above, we associate every element $\mathfrak{s}$ in $S_n$ with a point $\bs{x}$ in $\mathbb{T}^d$. For every $n \in \mathbb{N}$, we define $P^n_\Psi : S_n \to \mathbb{T}^d$ as
\begin{eqnarray}
    P^n_\Psi(\mathfrak{s}) \coloneqq \left(\frac{1}{2} + \frac{1}{4} \sum_{i = 1}^n s_i^1 \ell^{i-1}_\Psi \text{{\Large,}} \; \dots \; \text{{\Large,}} \; \frac{1}{2} + \frac{1}{4} \sum_{i = 1}^n s_i^d \ell^{i-1}_\Psi \right) \qquad \text{for} \qquad \mathfrak{s} = (\bs{s}_1, \dots , \bs{s}_n) \in S_n.
    \label{def: Pn Psi new}
\end{eqnarray}
The positions $P^n_\Psi(\mathfrak{s})$ would denote the centers of the $n$-th generation cubes in the Cantor set construction process (see figure \ref{fig: Cantor set}), and $P^n_\Theta(\mathfrak{s})$ denote the centers of the $n$-th generation dyadic cubes.
We also define $\wt{P}^n_\Psi : S_n \to \mathbb{T}^d$ as $\wt{P}^1_\Psi \coloneqq P^1_\Psi$ and 
\begin{align}
    \wt{P}^n_\Psi(\mathfrak{s}) & \coloneqq P^n_\Psi(\mathfrak{s}) - P^{n-1}_\Psi(\mathfrak{s}^\prime) + P^{n-1}_\Theta(\mathfrak{s}^\prime) \nonumber \\
    & = \left(\frac{1}{2} + \frac{1}{4} \sum_{i = 1}^{n-1} \frac{s_i^1}{2^{i-1}} + \frac{1}{4} s^1_n \ell^{n-1}_\Psi \text{{\Large,}} \; \dots \; \text{{\Large,}} \; \frac{1}{2} + \frac{1}{4} \sum_{i = 1}^{n-1} \frac{s_i^d}{2^{i-1}} + \frac{1}{4} s^d_n \ell^{n-1}_\Psi \right) \qquad \text{for} \quad n \geq 2.
    \label{def: Pn tilde}
\end{align}
We will need $\wt{P}^n_\Psi$ while constructing the vector field, which will be based on translating various cubes (see section \ref{section: overview}). In that respect, $\wt{P}^n_\Psi(\mathfrak{s})$ represents the center of an $n$-th generation cube in the Cantor set construction process relative to the center $P^{n-1}_\Psi(\mathfrak{s}^\prime)$ of a cube  from the previous generation, which is then shifted to $P^{n-1}_\Theta(\mathfrak{s}^\prime)$, the center of an $(n-1)$th-generation dyadic cube.
Finally, we define $\mathcal{P}_\Psi : \mathcal{S} \to \mathbb{T}^d$ as
\begin{eqnarray}
    \mathcal{P}_\Psi(\bs{\mathfrak{s}}) \coloneqq \left(\frac{1}{2} + \frac{1}{4} \sum_{i = 1}^\infty s_i^1 \ell^{i-1}_\Psi \text{{\Large,}} \; \dots \; \frac{1}{2} + \frac{1}{4} \sum_{i = 1}^\infty s_i^d \ell^{i-1}_\Psi \right).
\end{eqnarray}
 \begin{figure}[h]
\centering
 \includegraphics[scale = 0.65]{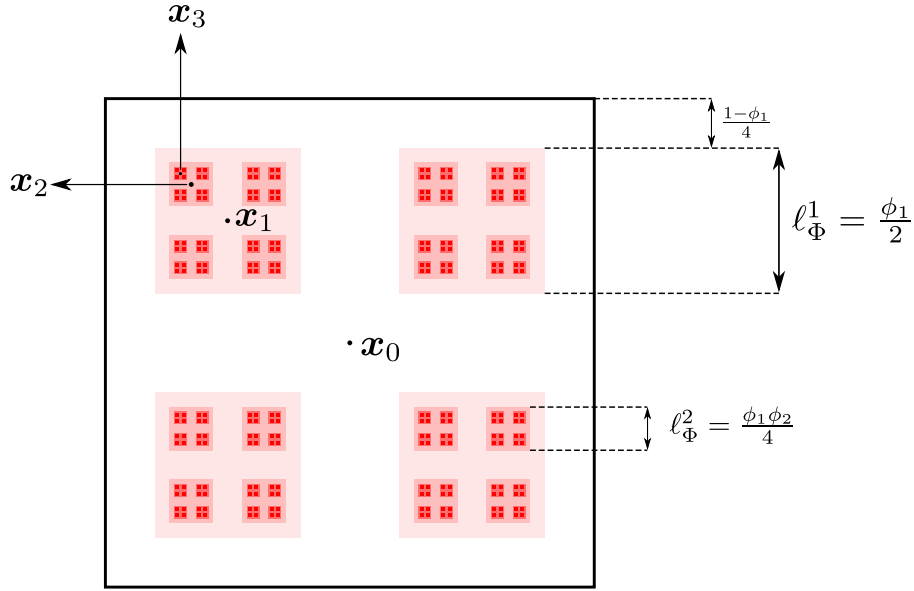}
 \caption{An illustration of the construction process of the Cantor set $\mathcal{C}_\Phi$ in two dimensions. The sequence $\Phi$ is given by (\ref{def: seq Phi}), and we have chosen $\nu = \frac{3}{4}$. The figure shows the collection of first four generations of cubes, $C^1_\Phi$, $C^2_\Phi$, $C^3_\Phi$ and $C^4_\Phi$, in increasingly less pale red color. In the figure, $\bs{x}_0 = (1/2, 1/2)$, $\bs{x}_1 = P^1_\Phi(\mathfrak{s}_1)$, $\bs{x}_2 = P^2_\Phi(\mathfrak{s}_2)$ and $\bs{x}_3 = P^3_\Phi(\mathfrak{s}_3)$. Here, $\mathfrak{s}_1$, $\mathfrak{s}_2$ and $\mathfrak{s}_3$ are the elements of $S_1$, $S_2$ and $S_3$, respectively, given by $\mathfrak{s}_1 = \{(-1, -1)\}$, $\mathfrak{s}_2 = \{(-1, -1), (-1, -1)\}$ and $\mathfrak{s}_3 = \{(-1, -1), (-1, -1), (-1, -1)\}$.}
 \label{fig: Cantor set}
\end{figure}
Next, we state a few useful lemmas whose proof is fairly straightforward and given here only for completeness.
\begin{lemma}
Let $\Psi = \{\psi_i\}_{i=1}^\infty$ be a sequence with $0 < \psi_i < 1$ and let $\mathfrak{s}$ and $\mathfrak{r}$ be two different elements of $S_n$ then the closed cubes $\ol{Q}(P^n_\Psi(\mathfrak{s}), \ell^n_\Psi)$ and $\ol{Q}(P^n_\Psi(\mathfrak{r}), \ell^n_\Psi)$ are disjoint.
\label{lemma: disjoint closed cubes}
\end{lemma}
\begin{proof}
Since $\mathfrak{s} = (\bs{s}_1, \dots, \bs{s}_n)$ and $\mathfrak{r} = (\bs{r}_1, \dots, \bs{r}_n)$ are different, that means for some $1 \leq i \leq n$, $\bs{s}_i = (s^1_i, \dots, s^d_i)$ and $\bs{r}_i = (r^1_i, \dots, r^d_i)$ are different. This, in turn, implies that for some $1 \leq j \leq d$, $s^j_i$ and $r^j_i$ are different. From the definition of $P^n_\Psi$ in (\ref{def: Pn Psi new}), we see that
\begin{eqnarray}
    \left|\left(P^n_\Psi(\mathfrak{s}) - P^n_\Psi(\mathfrak{r})\right)_j\right| \geq \frac{1}{2} \ell^{n-1}_\Psi, \nonumber
\end{eqnarray}
where $(\;\cdot\;)_j$ denotes the $j$-th component. Now suppose that $\bs{x} \in \ol{Q}(P^n_\Psi(\mathfrak{s}), \ell^n_\Psi)$, this means
\begin{eqnarray}
    \left|\left(P^n_\Psi(\mathfrak{s}) - \bs{x}\right)_j\right| \leq \frac{\ell^n_\Psi}{2},
    \label{eqn: inside close cube cond}
\end{eqnarray}
which implies
\begin{eqnarray}
    \left|\left(P^n_\Psi(\mathfrak{r}) - \bs{x}\right)_j\right| \geq \left|\left(P^n_\Psi(\mathfrak{s}) - P^n_\Psi(\mathfrak{r})\right)_j\right| - \left|\left(P^n_\Psi(\mathfrak{s}) - \bs{x}\right)_j\right|  \geq \frac{1}{2} \ell^{n-1}_\Psi - \frac{1}{2}\ell^n_\Psi \geq \frac{1}{\psi_i} \ell^{n}_\Psi - \frac{1}{2}\ell^n_\Psi >  \frac{1}{2}\ell^n_\Psi. \nonumber
\end{eqnarray}
Therefore, $\bs{x} \not\in \ol{Q}(P^n_\Psi(\mathfrak{r}), \ell^n_\Psi)$, which then completes the proof.
\end{proof}
\begin{lemma}
    Let the sequence $\Psi$ be the same as in the previous lemma, and let $\bs{\mathfrak{s}}$ and $\bs{\mathfrak{r}}$ be two different elements of $\mathcal{S}$. Then $\mathcal{P}_\Psi(\bs{\mathfrak{s}}) \neq \mathcal{P}_\Psi(\bs{\mathfrak{r}})$.
    \label{lemma: not equal infinite seq points}
\end{lemma}
\begin{proof}
    The proof is similar to the previous lemma.
\end{proof}
\begin{lemma}
Let $\Theta$ be a sequence of $1's$ as defined in (\ref{def: Theta seq of 1s}). Let $\mathfrak{s}$ and $\mathfrak{r}$ be two different elements of $S_n$ then the open cubes $Q(P^n_\Theta(\mathfrak{s}), \ell^n_\Theta)$ and $Q(P^n_\Theta(\mathfrak{r}), \ell^n_\Theta)$ are disjoint.
\label{lemma: disjoint open cubes}
\end{lemma}
\begin{proof}
The proof works the same as in Lemma \ref{lemma: disjoint closed cubes}. The only difference is that instead of (\ref{eqn: inside close cube cond}), for $\bs{x}$ to be in $Q(P^n_\Theta(\mathfrak{s}), \ell^n_\Theta)$, we need
\begin{eqnarray}
        \left|\left(P^n_\Theta(\mathfrak{s}) - \bs{x}\right)_j\right| < \frac{\ell^n_\Theta}{2}. 
\end{eqnarray}
\end{proof}

\begin{lemma}
Let $\Psi = \{\psi_i\}_{i=1}^\infty$ be a sequence with $0 < \psi_i \leq 1$ and let $\mathfrak{s} = (\bs{s}_1, \dots \bs{s}_n) \in S_n$ for some $n \geq 2$ and $\mathfrak{s}^\prime = (\bs{s}_1, \dots \bs{s}_{n-1})$. Then $\ol{Q}(P^n_\Psi(\mathfrak{s}), \ell^n_\Psi) \subset \ol{Q}(P^{n-1}_\Psi(\mathfrak{s}^\prime), \ell^{n-1}_\Psi)$.
\label{lemma: nested cubes}
\end{lemma}
\begin{proof}
    Suppose $\bs{x} \in \ol{Q}(P^n_\Psi(\mathfrak{s}), \ell^n_\Psi)$. This means
    \begin{eqnarray}
    \left|\left(P^n_\Psi(\mathfrak{s}) - \bs{x}\right)_j\right| \leq \frac{\ell^n_\Psi}{2}, \nonumber
\end{eqnarray}
where $(\;\cdot\;)_j$ denotes the $j$-th component. A straightforward calculation shows that
\begin{eqnarray}
 \left|\left(P^{n-1}_\Psi(\mathfrak{s}^\prime) - \bs{x}\right)_j\right| \leq \left|\left(P^{n-1}_\Psi(\mathfrak{s}^\prime) - P^n_\Psi(\mathfrak{s})\right)_j\right| + \left|\left(P^n_\Psi(\mathfrak{s}) - \bs{x}\right)_j\right| \leq \frac{\ell^{n-1}_\Psi}{4} + \frac{\ell^n_\Psi}{2} \leq \frac{\ell^{n-1}_\Psi}{2}. \nonumber
\end{eqnarray}
This implies $\bs{x} \in \ol{Q}(P^{n-1}_\Psi(\mathfrak{s}^\prime), \ell^{n-1}_\Psi)$. Hence, $\ol{Q}(P^n_\Psi(\mathfrak{s}), \ell^n_\Psi) \subset \ol{Q}(P^{n-1}_\Psi(\mathfrak{s}^\prime), \ell^{n-1}_\Psi)$ is true.
\end{proof}

\begin{lemma}
Let $\bs{\mathfrak{s}} \in \mathcal{S}$. Then for every $n \in \mathbb{N}$, $\mathcal{P}_\Psi(\bs{\mathfrak{s}}) \in \ol{Q}(P^n_\Psi(\sigma_n(\bs{\mathfrak{s}})), \ell^n_\Psi)$.
\label{lemma: seq lie in every cube}
\end{lemma}
\begin{proof}
For any $j \in \{1, \dots, d\}$, we note that
\begin{eqnarray}
   \left|\left(\mathcal{P}_\Psi(\bs{\mathfrak{s}}) - P^n_\Psi(\mathfrak{s})\right)_j\right| \leq \frac{1}{4}  \sum_{i = n + 1}^{\infty} \ell^{i-1}_\Psi \leq \frac{1}{4}  \sum_{i = n + 1}^{\infty} \frac{\ell^{n}_\Psi}{2^{i-n-1}} \leq \frac{\ell^{n}_\Psi}{2}. \nonumber
\end{eqnarray}
\end{proof}
\begin{corollary}
Let $\bs{\mathfrak{s}} \in \mathcal{S}$. Then  $\mathcal{P}_\Psi(\bs{\mathfrak{s}}) = \bigcap\limits_{n \in \mathbb{N}} \ol{Q}(P^n_\Psi(\sigma_n(\bs{\mathfrak{s}})), \ell^n_\Psi)$.
\label{lemma: another rep. of P psi s}
\end{corollary}
\begin{proof}
From Lemma \ref{lemma: nested cubes}, $\ol{Q}(P^n_\Psi(\sigma_n(\bs{\mathfrak{s}})), \ell^n_\Psi)$ forms a nested sequence of cubes, i.e.,
\begin{eqnarray}
    \ol{Q}(P^1_\Psi(\sigma_1(\bs{\mathfrak{s}})), \ell^1_\Psi) \supset \ol{Q}(P^2_\Psi(\sigma_2(\bs{\mathfrak{s}})), \ell^2_\Psi) \supset \dots. \nonumber
\end{eqnarray}
Also, the size of the cube $\ol{Q}(P^n_\Psi(\sigma_n(\bs{\mathfrak{s}})), \ell^n_\Psi)$ goes to zero as $n \to \infty$. Therefore, $\bigcap\limits_{n \in \mathbb{N}} \ol{Q}(P^n_\Psi(\sigma_n(\bs{\mathfrak{s}})), \ell^n_\Psi)$ contains only one single point from $\mathbb{T}^d$. From Lemma \ref{lemma: seq lie in every cube}, we note that $\mathcal{P}_\Psi(\bs{\mathfrak{s}}) \in \bigcap\limits_{n \in \mathbb{N}} \ol{Q}(P^n_\Psi(\sigma_n(\bs{\mathfrak{s}})), \ell^n_\Psi)$, which finishes the proof.
\end{proof}
\begin{lemma}
Let $\mathfrak{s} \in S_n$, then $\ol{Q}(P^n_\Theta(\mathfrak{s}), \ell^n_\Theta) = \bigcup\limits_{\substack{\mathfrak{r} \in S_{n+1} \\ \mathfrak{r}^\prime = \mathfrak{s} }} \ol{Q}(P^{n+1}_\Theta(\mathfrak{r}), \ell^{n+1}_\Theta)$. Furthermore, $\mathbb{T}^d = \bigcup\limits_{\mathfrak{s} \in S_n}\ol{Q}(P^n_\Theta(\mathfrak{s}), \ell^n_\Theta)$ for any $n \in \mathbb{N}$.
\label{lemma: Torus union of dyadic cubes}
\end{lemma}
\begin{proof}
    These are standard properties in a dyadic decomposition.
\end{proof}
\begin{lemma}
For every $\bs{x} \in \mathbb{T}^d$, $\exists \; \bs{\mathfrak{s}} \in \mathcal{S}$ such that $\bs{x} = \mathcal{P}_\Theta(\bs{\mathfrak{s}})$.
\label{lemma: Torus d s}
\end{lemma}
\begin{proof}
    From Lemma \ref{lemma: Torus union of dyadic cubes}, there exist $\mathfrak{s}_1 \in S_1$ such that $\bs{x} \in \ol{Q}(P^1_\Theta(\mathfrak{s}_1), \ell^1_\Theta)$. Having selected $\mathfrak{s}_i \in S_i$ for $i \geq 1$ such that $\bs{x} \in \ol{Q}(P^i_\Theta(\mathfrak{s}_i), \ell^i_\Theta)$, again using Lemma \ref{lemma: Torus union of dyadic cubes}, we choose $\mathfrak{s}_{i+1} \in S_{i+1}$ such that $\mathfrak{s}_{i+1}^\prime = \mathfrak{s}_i$ and $\bs{x} \in \ol{Q}(P^{i+1}_\Theta(\mathfrak{s}_{i+1}), \ell^{i+1}_\Theta)$. Finally, we define $\bs{\mathfrak{s}} \in \mathcal{S}$ as follows. We let the $n$th component of $\bs{\mathfrak{s}}$ to be the $n$th component of $\mathfrak{s}_n$, i.e., $\bs{s}_n = \bs{s}_{n, n}$. By construction $\bs{x} = \bigcap\limits_{n \in \mathbb{N}} \ol{Q}(P^n_\Psi(\sigma_n(\bs{\mathfrak{s}})), \ell^n_\Psi)$. Finally, referring to Corollary \ref{lemma: another rep. of P psi s} finishes the proof.
\end{proof}

\noindent
\textbf{Notation 4.} 
If $\Psi = \{\psi_i\}_{i = 1}^\infty$ is such that $0 < \psi_i < 1$, then we define
\begin{eqnarray}
    C^n_\Psi \coloneqq \bigcup\limits_{\mathfrak{s} \in S_n} \overline{Q}(P^n_\Psi(\mathfrak{s}), \ell^n_\Psi).
\end{eqnarray}
The set $C^n_\Psi$ is the union of $n$th generation cubes in the Cantor set construction process. From Lemma \ref{lemma: nested cubes}, we note that $C^n_\Psi$'s form a nested sequence of sets, i.e.,
\begin{eqnarray}
    C^1_{\Psi} \supset C^{2}_{\Psi} \supset C^{3}_{\Psi} \dots.
\end{eqnarray}
Finally, we define a Cantor set corresponding to the sequence $\Psi$ as
\begin{eqnarray}
    \mathcal{C}_\Psi \coloneqq \bigcap\limits_{i = 1}^{\infty} C^n_\Psi.
\label{def: Cantor set}
\end{eqnarray}
From the definition (\ref{def: Cantor set}) and Corollary \ref{lemma: another rep. of P psi s}, we immediately see that for any $\bs{\mathfrak{s}} \in \mathcal{S}$, $\mathcal{P}_\Psi(\bs{\mathfrak{s}}) \in \mathcal{C}_\Psi$. Moreover, the following lemma states that any $\bs{x} \in \mathcal{C}_\Psi$ can be represented this way.



\begin{lemma}
Let $\Psi = \{\psi_i\}_{i=1}^\infty$ be a sequence with $0 < \psi_i < 1$, then for every $\bs{x} \in \mathcal{C}_\Psi$, $\exists! \bs{\mathfrak{s}} \in \mathcal{S}$ such that $\bs{x} = \mathcal{P}_\Psi(\bs{\mathfrak{s}})$.
\label{lemma: C Psi unique s}
\end{lemma}
\begin{proof}
    We only need to show the existence of $\bs{\mathfrak{s}} \in \mathcal{S}$ as the uniqueness is clear from Lemma \ref{lemma: not equal infinite seq points}. Given $\bs{x} \in \mathcal{C}_\Psi$, we construct an $\bs{\mathfrak{s}} = (\bs{s}_1, \bs{s}_2, \dots) \in \mathcal{S}$ as follows.

    By definition of $\mathcal{C}_\Psi$, for every $n \in \mathbb{N}$ there is $\mathfrak{s}_n \in S_n$ such that $\bs{x} \in \ol{Q}(P^n_\Psi(\mathfrak{s}_n), \ell^n_\Psi)$. Moreover, as the cubes from the $n$th generation in the Cantor set construction process are disjoint from Lemma \ref{lemma: disjoint closed cubes}, this $\mathfrak{s}_n$ is unique. Using (\ref{def: prime notation mathfrak s}) from Notation 2, this also means $\mathfrak{s}_n^\prime = \mathfrak{s}_{n-1}$. Finally, we define the $n$th component of $\bs{\mathfrak{s}}$ as $\bs{s}_n \coloneqq \bs{s}_{n, n}$, where $\bs{s}_{n, n}$ is the $n$th component of $\mathfrak{s}_n$.

    By construction of $\bs{\mathfrak{s}}$, we see that for every $n \in \mathbb{N}$, $\bs{x}  \in \ol{Q}(P^n_\Psi(\sigma_n(\bs{\mathfrak{s}})), \ell^n_\Psi)$. Finally, noting Corollary \ref{lemma: another rep. of P psi s} implies $\bs{x} = \mathcal{P}(\bs{\mathfrak{s}})$.
\end{proof}

An important sequence that we use for the construction of our vector field is
\begin{eqnarray}
    \Phi \coloneqq \{\phi_i\}_{i=1}^\infty \coloneqq \left\{\frac{1}{2^{\nu}}\right\}_{i=1}^\infty \qquad \text{where } \nu \in (0, 1) \text{ is a constant}.
    \label{def: seq Phi}
\end{eqnarray}
From this sequence $\Phi$, we have that
\begin{eqnarray}
   \ell^0_\Phi = 1 \qquad \qquad \text{and} \qquad \qquad \ell^n_\Phi = \frac{1}{2^{(1+\nu)n}} \qquad \text{for} \qquad n \in \mathbb{N}.
\end{eqnarray}
\begin{lemma}
    Let the sequence $\Phi$ be as given in (\ref{def: seq Phi}). Then the Hausdorff dimension of the corresponding Cantor set $\mathcal{C}_\Phi$ is $$\dim_{H} \mathcal{C}_\Phi = \frac{d}{1+\nu}.$$ 
    \label{lemma: Hausdorff dimension}
\end{lemma}
\begin{proof}
    The calculation of the Hausdorff dimension of a Cantor set is standard, see for example, \cite[ch. 7]{SteinShakarchiRealAnalysis} or \cite[ch. 3]{FalconerFractal}.
\end{proof}

Finally, we finish this section by stating a simple lemma that will be useful in proof of Proposition \ref{Prop of vi} given in section \ref{subsection: assembly of moving blobs}. We first let
\begin{eqnarray}
    \vartheta = \vartheta(\nu) \coloneqq \frac{1}{8}(2^\nu - 1).
    \label{def: varepsilon}
\end{eqnarray}
\begin{lemma}
For any $a \in [0, 1]$ and for any $\mathfrak{s} \in S_i$ for some $i \in \mathbb{N}$, let $\bs{x}_a = (1-a) \wt{P}^i_\Phi(\mathfrak{s}) + a P^i_\Theta(\mathfrak{s})$ then
\begin{eqnarray}
    \ol{Q}(\bs{x}_a, (1 + \vartheta) \ell^i_\Phi) \Subset Q(P^i_\Theta(\mathfrak{s}), \ell^i_\Theta).  \nonumber
\end{eqnarray}
\label{lemma: an important subset}
\end{lemma}
\begin{proof}
    The proof is a straightforward verification of the statement.
\end{proof}


\section{Overview of the approach}
\label{section: overview}
In order to prove Theorem \ref{thm unsteady}, we aim to construct a vector field $\bs{u}$  such that there exists a flow map for which, at some time, $T > 0$, we have
\begin{eqnarray}
   X^{\bs{u}}(T, \mathbb{T}^d) \subseteq \mathcal{C}_\Phi.
   \label{cond: X u crushing}
\end{eqnarray}
From Lemma \ref{lemma: Hausdorff dimension}, we see that the Hausdorff dimension $\dim_H \mathcal{C}_\Phi < d$. Therefore, the $d$-dimensional Lebesgue measure of $\mathcal{C}_\Phi$ is zero. As a result, this flow map $X^{\bs{u}}$ is not a regular Lagrangian flow. However, as we shall see, the flow $\bs{u}$ possesses a Sobolev regularity and falls in the range of DiPerna and Lions theory, which then guarantees the existence of a regular Lagrangian flow $X^{\bs{u}}_{\textit{\tiny RL}}$ as well. The existence of two different flow maps implies that the set of initial conditions for which the trajectories are not unique is a set of positive measure. Because the condition (\ref{cond: X u crushing}) holds, this set is, in fact, a full-measure set. 

 \begin{figure}[H]
\centering
 \includegraphics[scale = 0.43]{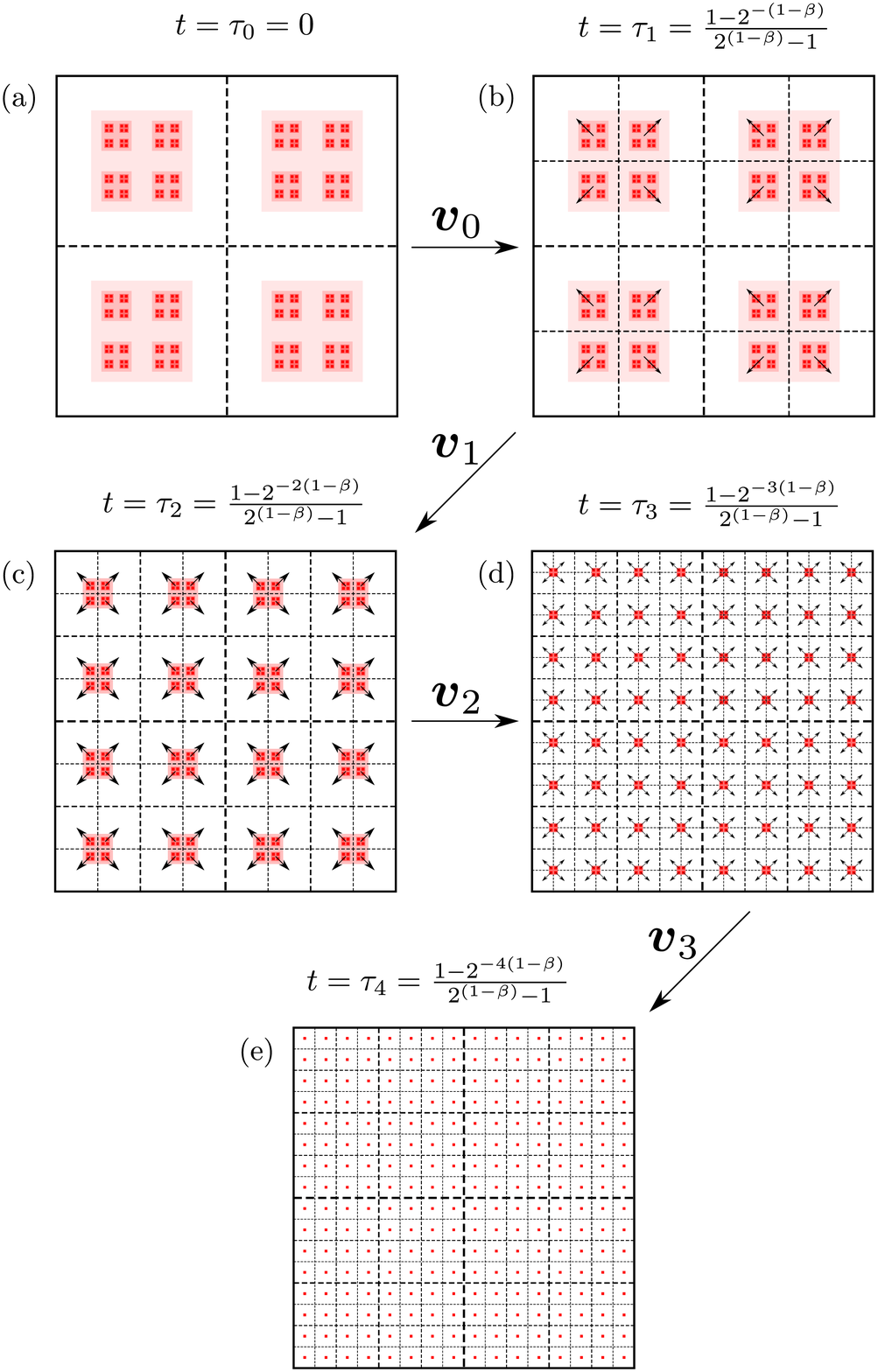}
 \caption{The motion of the first four generations of cubes in the Cantor set construction process under the flow of vector field $\bs{v} = \sum_{i=0}^\infty \bs{v}_i$. At the $i$th stage, the vector field $\bs{v}_i$ translates the cubes of the $i$th generation so that their centers align with the dyadic cubes of the $i$th generation. In the first step, by definition of the centers (\ref{def: Pn Psi new}), we have this alignment at $t = \tau_0$ itself. Therefore, in the first step, we simply choose $\bs{v}_0 \equiv 0$. The arrows in the figure indicate the direction of the motion of cubes. In panels (b), (c) and (d), the tails of the arrow lie at the shifted center (see (\ref{def: Pn tilde})), $\wt{P}^1_\Phi(\mathfrak{s}_1)$, $\wt{P}^2_\Phi(\mathfrak{s}_2)$ and $\wt{P}^3_\Phi(\mathfrak{s}_3)$, respectively, for $\mathfrak{s}_1 \in S_1$, $\mathfrak{s}_2 \in S_2$ and $\mathfrak{s}_3 \in S_3$. Finally, in our construction, the vector field $\bs{v}_{i-1}$ is supported near the $i$th generation cubes in the Cantor set construction process as they move around in space. The size of these $i$th generation cubes becomes increasingly smaller compared to $i$th generation dyadic cubes, as can be seen in panel (e), for example. This shrinking of the support of the vector field $\bs{v}_i$ with large $i$ is one of the main reasons that allow us to bound the Sobolev norm of the vector field uniformly in time.}
 \label{fig: Time evolution}
\end{figure}

We construct the vector field $\bs{u}$ through a time reversal argument applied to a vector field $\bs{v}$, whose construction we describe next. For some $0 < \beta < 1$, let us define a sequence of times as
\begin{eqnarray}
    \tau_i \coloneqq \frac{1 - 2^{-(1-\beta)i}}{2^{(1-\beta)} - 1} \quad \text{for} \quad i \in \mathbb{Z}_{\geq 0} \qquad \text{and} \qquad \tau_\infty \coloneqq \frac{1}{2^{(1-\beta)} - 1}. 
    \label{def: tau i}
\end{eqnarray}
It is clear that
\begin{eqnarray}
    \tau_{i}-\tau_{i-1} = \frac{1}{2^{(1-\beta)i}} \quad \text{for} \quad i \in \mathbb{N}. \nonumber
\end{eqnarray}

We design the vector field $\bs{v}$ to be such that it has a unique flow map $X^{\bs{v}}$. Moreover, at time $t = \tau_\infty$, we have
\begin{eqnarray}
   X^{\bs{v}}(\tau_\infty, \mathcal{C}_\Phi) = \mathbb{T}^d.
   \label{cond: X v expanding}
\end{eqnarray}
In our definition of the vector field $\bs{v}$, we will write $\bs{v}$ to be an infinite sum of vector fields $\bs{v}_i$'s, i.e., $\bs{v} = \sum_{i=0}^\infty \bs{v}_i$. Under the flow of vector field $\bs{v}$, the mapping of $\mathcal{C}_\Phi$ to $\mathbb{T}^d$ will occur in a sequence of infinite steps. The vector field $\bs{v}_{i-1}$, whose support lies in $[\tau_{i-1}, \tau_{i}]$, will execute the $i$th step in the sequence. Figure \ref{fig: Time evolution} depicts the first four steps in this infinite sequence of steps.

In the first step, the vector field $\bs{v}_0$ translates $Q(P^1_\Phi(\mathfrak{s}), \ell^1_\Phi)$ (the first-generation cubes in the Cantor set construction process) such that after the translation, the centers of these cubes align with the centers of first-generation dyadic cubes.
Continuing in this way, at the $i$th step,  the vector field $\bs{v}_{i-1}$ translates the cubes $Q(\wt{P}^{i}_\Phi(\mathfrak{s}), \ell^i_\Phi)$ (the $i$th generation cubes in the Cantor set construction process after translations under flow $\bs{v}_j$'s, $j < i-1$) such that their centers align with the centers of the $i$th generation dyadic cubes.
After performing the $i$th step, the centers of the $i+1$th generation of cubes in the Cantor set construction process now shift to $\wt{P}^{i+1}_\Phi(\mathfrak{s})$, which we note are different from the centers $P^{i+1}_\Phi(\mathfrak{s})$ in the original configuration of the Cantor set. To translate the cubes at any step, we use what we call a ``blob flow." An overview and the construction of a blob flow are given in section \ref{sec: blob flow}.

Now we quickly see why the flow of vector field $\bs{v}$ takes $\mathcal{C}_\Phi$ to $\mathbb{T}^d$.
From Lemma \ref{lemma: Torus d s}, we see that for every $\bs{x}_e \in \mathbb{T}^d$, there is an element $\bs{\mathfrak{s}} \in \mathcal{S}$ such that $\bs{x}_e = \mathcal{P}_\Theta(\bs{\mathfrak{s}})$. But for this $\bs{\mathfrak{s}}$ there is $\bs{x}_s \in \mathcal{C}_\Phi$ given by $\bs{x}_s = \mathcal{P}_\Phi(\bs{\mathfrak{s}})$. From the description given above, a trajectory corresponding to vector field $\bs{v}$ starting at $\bs{x}_s$ after time $\tau_i$ is given by $\gamma^{\bs{v}}_{\bs{x}_s}(\tau_i) = \mathcal{P}_\Phi(\bs{\mathfrak{s}}) - P^i_\Phi(\sigma_i(\bs{\mathfrak{s}})) + P^i_\Theta(\sigma_i(\bs{\mathfrak{s}}))$. Letting $i \to \infty$, we see that $\gamma^{\bs{v}}_{\bs{x}_s}(\tau_\infty) = \mathcal{P}_\Theta(\bs{\mathfrak{s}})$. In conclusion, $X^{\bs{v}}(\tau_\infty, \mathcal{C}_\Phi) = \mathbb{T}^d$.



Finally, we note that the vector field $\bs{v}$ that we create lies in $C([0, \tau_\infty]; W^{1,p}(\mathbb{T}^d, \mathbb{R}^d))$, and we can do this for any $1 \leq p < d$. Our vector field design has one main advantage compared to, for example, the ``checkerboard" flow used in optimal mixing problems \cite{YaoZlatos17, AlbertiCrippaMazzucato19}. For both a checkerboard flow and our vector field $\bs{v}$, the number of cubes at the $i$th step are the same, which is $2^{id}$. However, the size of cubes that we translate at the $i$th step in our vector field $\bs{v}$ is $\frac{1}{2^{(1+\nu)i}}$, which is at large $i$ substantially smaller than $\frac{1}{2^i}$, the size of cubes at the $i$th step in a typical checkerboard flow. This is one of the main reasons we are able to bound the Sobolev norm uniformly in time.

\section{Proof of Theorem \ref{thm unsteady} and the design of vector field $\bs{v}$}
\label{sec: proof of thms}
This section will prove Theorem \ref{thm unsteady} and construct the vector field $\bs{v}$ from last section, given the properties of its components $\bs{v}_i$. Before stating the next proposition, recall the definition of $\Phi$ from (\ref{def: seq Phi}), which uses a parameter $\nu$ and the definition of $\tau_i$ from (\ref{def: tau i}), which uses a parameter $\beta$.

\begin{proposition}
\label{Prop of v}
For every $p < d$ and $\alpha < 1$, there are two numbers, $\nu \coloneqq \nu(p, \alpha) \in (0, 1)$ and $\beta \coloneqq 1 - \nu^2$, and there exists a divergence-free vector field $\bs{v} \in C([0, \tau_\infty]; W^{1, p}(\mathbb{T}^d, \mathbb{R}^d)) \cap C^{\alpha}([0, \tau_\infty] \times \mathbb{T}^d, \mathbb{R}^d) \cap C^\infty([0, \tau_\infty) \times \mathbb{T}^d, \mathbb{R}^d)$ with the following two properties: \\
(i) For every $\bs{x}_s \in \mathbb{T}^d$, the trajectory $\gamma_{\bs{x}_s}^{\bs{v}} : [0, \tau_\infty] \to \mathbb{T}^d$ is unique. As a result, the flow map $X^{\bs{v}}$ is also unique. \\
(ii) Let $\Phi$ be as in (\ref{def: seq Phi}). Then given a sequence $\bs{\mathfrak{s}} \in \mathcal{S}$, let $\bs{x}_s = \mathcal{P}_\Phi(\bs{\mathfrak{s}})$ and $\bs{x}_e = \mathcal{P}_\Theta(\bs{\mathfrak{s}})$, then $\gamma_{\bs{x}_s}^{\bs{v}}(\tau_\infty) = \bs{x}_e$. \\
\end{proposition}

\begin{proof}[Proof of Theorem \ref{thm unsteady}]
Let's define the required vector field $\bs{u} : [0, T] \times \mathbb{T}^d \to \mathbb{R}^d$ as
\begin{eqnarray}
    \bs{u}(t, \bs{x}) \coloneqq - \frac{\tau_\infty}{T} \bs{v} \left( \tau_\infty \left(1 - \frac{t}{T}\right), \bs{x}\right). \nonumber
\end{eqnarray}
Clearly, $\bs{u} \in C([0, T]; W^{1, p}(\mathbb{T}^d, \mathbb{R}^d)) \cap C^{\alpha}([0, T] \times \mathbb{T}^d, \mathbb{R}^d)$ and is divergence-free.

Next, for every $\bs{x} \in \mathbb{T}^d$, we define a trajectory corresponding to the vector field $\bs{u}$ that starts at $\bs{x}$. From Lemma \ref{lemma: Torus d s}, for every $\bs{x}_e \in \mathbb{T}^d$, there exist an $\bs{\mathfrak{s}} \in \mathcal{S}$ such that $\bs{x}_e = \mathcal{P}_\Theta(\bs{\mathfrak{s}})$. After choosing such an $\bs{\mathfrak{s}}$, let us assign $\bs{x}_s = \mathcal{P}_\Phi(\bs{\mathfrak{s}})$. With that, it is easy to verify that $\gamma^{\bs{u}}_{\bs{x}_e}:[0, T] \to \mathbb{T}^d$ defined as
\begin{eqnarray}
    \gamma^{\bs{u}}_{\bs{x}_e} (t) \coloneqq \gamma^{\bs{v}}_{\bs{x}_s} \left(\tau_\infty \left(1 - \frac{t}{T}\right)\right) \nonumber
\end{eqnarray}
is a trajectory corresponding to the vector field $\bs{u}$ starting from $\bs{x}_e$.

Finally, we define a map $X^{\bs{u}}:[0,T]\times\mathbb{T}^d \to \mathbb{T}^d$ as 
\begin{eqnarray}
  X^{\bs{u}}(t, \bs{x}) \coloneqq \gamma^{\bs{u}}_{\bs{x}}(t). \nonumber  
\end{eqnarray}
Clearly, $X^{\bs{u}}$ is a flow map for which 
\begin{eqnarray}
    X^{\bs{u}}(T, \mathbb{T}^d) \subseteq \mathcal{C}_\Phi.
    \label{X u crusing condition 2}
\end{eqnarray}
As the Hausdorff dimension $\dim_H \mathcal{C}_\Phi < d$, which means $\mathscr{L}^d(\mathcal{C}_\Phi) = 0$. Therefore, $X^{\bs{u}}$ is not a regular Lagrangian flow. As the vector field $\bs{u}$ has the required Sobolev regularity, from DiPerna--Lions theory, there is another flow map $X^{\bs{u}}_{\textit{\tiny RL}}$ which is regular Lagrangian. This implies that the set of initial conditions for which the trajectories are not unique has a positive $d$-dimensional Lebesgue measure. Furthermore, from (\ref{X u crusing condition 2}) and the definition of regular Lagrangian flow (\ref{def: Regular Lagrangian flow}), we see that this set is, in fact, a full-measure set.
\end{proof}

\begin{proposition}
\label{Prop of vi}
Given $\nu \in (0,1)$ and $\beta \in (0, 1)$, for every $i \in \mathbb{Z}_{\geq 0}$, there exists a vector field $\bs{v}_i : [0, \tau_\infty] \times \mathbb{T}^d \to \mathbb{R}^d$ with the following properties. \\[5pt]
(i) $\bs{v}_i \in C^\infty([0, \tau_\infty] \times \mathbb{T}^d, \mathbb{R}^d) $, \\[5pt]
(ii) $\bs{v}_i$ is divergence-free, \\[5pt]
(iii) $\supp_t \bs{v}_i \Subset (\tau_i, \tau_{i+1})$, \\[5pt]
 (iv) $\norm{\bs{v}_i}_{L^\infty_{t, \bs{x}}} \lesssim \frac{1}{(2^{\nu}-1)} \frac{1}{2^{\beta i}}$ \\[5pt]
(v) $\norm{\nabla \bs{v}_i}_{L^\infty_{t, \bs{x}}} \lesssim \frac{2^{(1 + \nu - \beta) i}}{(2^{\nu}-1)^2}$ \\[5pt]
(vi) For a given $p \in [1, \infty)$,  $\norm{\bs{v}_i(t, \cdot)}_{W^{1, p}} \lesssim 
 \frac{1}{(2^\nu - 1)^2} \times 2^{\frac{[(1 + \nu - \beta)p - d \nu]}{p}i}$, \\[5pt]
 (vii) $\norm{\partial_t \bs{v}_i}_{L^\infty_{t, \bs{x}}} \lesssim \frac{1}{(2^\nu - 1)^2}  \frac{1}{2^{[2 \beta - \nu - 1]i}}$\\[5pt]
(viii) For a given $\mathfrak{s} \in S_{i+1}$ and $\bs{x} \in \ol{Q}(\wt{P}^{i+1}_\Phi(\mathfrak{s}), \ell^{i+1}_\Phi)$, we have $\gamma^{\bs{v}_i}_{\bs{x}}(\tau_{i+1}) = \bs{x} - \wt{P}^{i+1}_\Phi(\mathfrak{s}) + P^{i+1}_\Theta(\mathfrak{s}).$\\[5pt]
\end{proposition}
\begin{proof}[Proof of Proposition \ref{Prop of v}]
Let us define a vector field $\bs{v}:[0, \tau_\infty] \times \mathbb{T}^d \to \mathbb{R}^d$ as follows
\begin{eqnarray}
    \bs{v} \coloneqq \sum_{i = 0}^\infty \bs{v}_i,
\end{eqnarray}
where $\bs{v}_i$'s are the vector fields from Proposition \ref{Prop of vi}. Next, we show that the vector field $\bs{v}$ has the properties stated in Proposition \ref{Prop of v}. \\
\\
In the proof that follows, we make the following choices of parameters: $\nu$ is a sufficiently small number and $\beta = 1 - \nu^2$.\\
\\
\textbullet \;\underline{$\bs{v} \in C^\infty([0, \tau_\infty)\times\mathbb{T}^d, \mathbb{R}^d)$:}  We conclude this by noting property (i) and that the support of $\bs{v}_i$'s are disjoint.\\
\\
\textbullet \;\underline{$\bs{v}$ is divergence-free:} As for any time $t \in [0, \tau_\infty]$, at most one of the $\bs{v}_i$ is non-zero, $\nabla \cdot \bs{v} \equiv 0$ follows from the divergence-free nature of the $\bs{v}_i$'s.\\
\\
\textbullet \; \underline{$\bs{v} \in C([0, \tau_\infty]; W^{1, p}(\mathbb{T}^d, \mathbb{R}^d))$:} From the disjoint support of $\bs{v}_i$'s in time, we note that
\[\norm{\bs{v}}_{L^\infty_t W^{1, p}_{\bs{x}}} \leq \sup_{i} \norm{\bs{v}_i}_{L^\infty_t W^{1, p}_{\bs{x}}}.\]
From property (vi) in Proposition \ref{Prop of vi}, we see that the right-hand side is bounded if we choose
\begin{eqnarray}
  p < \frac{d \nu}{1 + \nu - \beta}.  
  \label{eqn: the cond on p}
\end{eqnarray}
As $\bs{v} \in C^\infty([0, \tau_\infty)\times\mathbb{T}^d, \mathbb{R}^d)$, the continuity of the Sobolev norm in time is clear for any $t \in [0, \tau_\infty)$. 
For $t = \tau_\infty$, we simply have 
\begin{eqnarray}
\norm{\bs{v}(\tau_\infty, \cdot) - \bs{v}(\tilde{t}, \cdot)}_{W^{1, p}} = \norm{\bs{v}(\tilde{t}, \cdot)}_{W^{1, p}}. \nonumber
\end{eqnarray}
Using the property (vi) in Proposition \ref{Prop of vi}, the right-hand side decreases to zero as $\tilde{t}$ approaches $\tau_\infty$ if (\ref{eqn: the cond on p}) holds. Finally, we note that in condition (\ref{eqn: the cond on p}), the exponent $p$ can be made arbitrarily close to $d$ if we choose $\nu$ small enough and $\beta = 1 - \nu^2$. \\
\\
\textbullet \; \underline{$\bs{v} \in C^\alpha([0, \tau_\infty] \times \mathbb{T}^d, \mathbb{R}^d)$:} We start with the definition of  $\alpha$-H\"older norm:
\begin{eqnarray}
    \norm{\bs{v}}_{C^\alpha_{t, \bs{x}}} = \sup_{(t_1, \bs{x}_1) \neq (t_2, \bs{x}_2)} \frac{|\bs{v}(t_1, \bs{x}_1) - \bs{v}(t_2, \bs{x}_2)|}{|(t_1, \bs{x}_1) - (t_2, \bs{x}_2)|^\alpha}.
    \label{eqn: Holder def}
\end{eqnarray}
If $t_1 = t_2 = \tau_\infty$, but $\bs{x}_1 \neq \bs{x}_2$, then the right-hand side in (\ref{eqn: Holder def}) is simply zero. If $t_1 = \tau_\infty$, but $t_2 \neq \tau_\infty$, then we have
\begin{eqnarray}
    \frac{|\bs{v}(t_1, \bs{x}_1) - \bs{v}(t_2, \bs{x}_2)|}{|(t_1, \bs{x}_1) - (t_2, \bs{x}_2)|^\alpha} = \frac{|\bs{v}(t_2, \bs{x}_2)|}{|(\tau_\infty, \bs{x}_1) - (t_2, \bs{x}_2)|^\alpha} \leq \frac{|\bs{v}(t_2, \bs{x}_2)|}{|\tau_\infty - t_2|^\alpha} \lesssim \frac{1}{2^\nu - 1} \sup_{i} \frac{1}{2^{[\beta - \alpha(1-\beta)]i}}, \nonumber
\end{eqnarray}
which is bounded if
\begin{eqnarray}
    \alpha < \frac{\beta}{1-\beta}.
\end{eqnarray}
Finally, if $t_1, t_2 \in [0, \tau_\infty)$, then we have
\begin{eqnarray}
    \frac{|\bs{v}(t_1, \bs{x}_1) - \bs{v}(t_2, \bs{x}_2)|}{|(t_1, \bs{x}_1) - (t_2, \bs{x}_2)|^\alpha} \leq \norm{\partial_t \bs{v}}_{L^\infty_{t, \bs{x}}} |t_1 - t_2|^{1 - \alpha} + \sup_{t} \norm{\bs{v}(t, \cdot)}_{C^\alpha}. \nonumber
\end{eqnarray}
Now the first term on the right-hand side is bounded if $\alpha \leq 1$ and $1 + \nu < 2 \beta$, which is true for the choices we made at the beginning of the proof. For the second term, we have
\begin{align}
    \sup_{\bs{x}_1 \neq \bs{x}_2}\frac{|\bs{v}(t, \bs{x}_1) - \bs{v}(t, \bs{x}_2)|}{| \bs{x}_1 - \bs{x}_2|^\alpha} & \leq 2 \min\left\{|\bs{x}_1 - \bs{x}_2|^{1-\alpha} \norm{\nabla \bs{v} (t, \cdot)}_{L^\infty}, \frac{\norm{\bs{v}(t, \cdot)}_{L^{\infty}}}{|\bs{x}_1 - \bs{x}_2|^\alpha}\right\}, \nonumber \\
    & \leq 2 \norm{\bs{v}(t, \cdot)}_{L^{\infty}}^{1-\alpha} \norm{\nabla \bs{v} (t, \cdot)}_{L^\infty}^{\alpha}, \nonumber \\
    & \lesssim \left(\frac{1}{2^\nu - 1}\right)^{1+\alpha}\sup_{i} \; 2^{[\alpha(1+\nu) - \beta]i}, \nonumber
\end{align}
which is bounded if 
\begin{eqnarray}
    \alpha < \frac{\beta}{1 + \nu}.
\end{eqnarray}
As before, by choosing $\nu$ to be small enough and $\beta = 1 - \nu^2$, the exponent $\alpha$ can be made arbitrarily close to one. \\
\\
\textbullet \; \underline{Uniqueness of trajectories:} As the vector field $\bs{v} \in C^\infty([0, \tau_\infty) \times \mathbb{T}^d, \mathbb{R}^d)$, the existence and uniqueness of a trajectory $\gamma^{\bs{v}}_{\bs{x}}$ starting from $\bs{x} \in \mathbb{T}^d$ is clear for any time $t < \tau_\infty$. The existence and uniqueness at time $t = \tau_\infty$ are obtained from the fact that $\norm{\bs{v}(t, \cdot)}_{L^\infty} \to 0$ as $t \to \tau_\infty$. \\
\\
\textbullet \; \underline{Trajectories starting from $\mathcal{C}_\Phi$:} Let $\bs{x}_s \in \mathcal{C}_\Phi$, then from Lemma \ref{lemma: C Psi unique s}, there is a unique $\bs{\mathfrak{s}} \in \mathcal{S}$ such that $\bs{x}_s = \mathcal{P}_\Phi(\bs{\mathfrak{s}})$. We will show that $\gamma^{\bs{v}}_{\bs{x}_s}(\tau_\infty) = \bs{x}_e$, where $\bs{x}_e = \mathcal{P}_\Theta(\bs{\mathfrak{s}})$, which will then imply $X^{\bs{v}}(\tau_\infty, \mathcal{C}_\Phi) = \mathbb{T}^d$ from Lemma \ref{lemma: Torus d s}.\\ Given a $\bs{x} \in \mathbb{T}^d$, let us first define $\bs{y}_i \in \mathbb{T}^d$ for all $i \in \mathbb{Z}_{\geq 0}$ as
\begin{eqnarray}
    \bs{y}_0 \coloneqq \bs{x}_s, \qquad \text{ and } \qquad \bs{y}_{i+1} \coloneqq \gamma^{\bs{v}_{i}}_{\bs{y}_i}(\tau_i) \quad \forall \, i \in \mathbb{Z}_{\geq 0}. \nonumber
\end{eqnarray}
Next, let us define a trajectory $\gamma^{\bs{v}}_{\bs{x}}:[0, \tau_\infty] \to \mathbb{T}^d$ as
\begin{eqnarray}
    \gamma^{\bs{v}}_{\bs{x}}(t) \coloneqq
    \begin{cases}
    \gamma^{\bs{v}_i}_{\bs{y}_i}(t) & \quad \text{for } t \in [\tau_i, \tau_{i+1}), \\[5pt]
    \lim\limits_{i \to \infty} \gamma^{\bs{v}_i}_{\bs{y}_i}(\tau_{i}) & \quad \text{if } t = \tau_\infty.
    \end{cases}
    \label{eqn: traj gamma v}
\end{eqnarray}
Using property (iii) about the disjoint supports of $\bs{v}_i$'s, one can verify through an induction argument that $\gamma^{\bs{v}}_{\bs{x}}$ is indeed the unique trajectory corresponding to the flow $\bs{v}$ starting at $\bs{x}$. Next, we claim that 
\begin{eqnarray}
    \gamma^{\bs{v}}_{\bs{x}_s}(\tau_i) = \mathcal{P}_{\Phi}(\bs{\mathfrak{s}}) - P^i_\Phi(\sigma_i(\bs{\mathfrak{s}})) + P^i_\Theta(\sigma_i(\bs{\mathfrak{s}})).
    \label{eqn: traj v claim}
\end{eqnarray}
The claim is obviously true for $i = 1$. Now suppose that the claim is true for some $i \geq 1$. We will show that it is also true for $i+1$. From Lemma \ref{lemma: seq lie in every cube}, we know that
\begin{eqnarray}
    \mathcal{P}_\Phi(\bs{\mathfrak{s}}) \in \ol{Q}(P_\Phi^{i+1}(\sigma_{i+1}(\bs{\mathfrak{s}})), \ell^{i+1}_\Phi). \nonumber
\end{eqnarray}
This implies  
\begin{eqnarray}
    \mathcal{P}_{\Phi}(\bs{\mathfrak{s}}) - P^i_\Phi(\sigma_i(\bs{\mathfrak{s}})) + P^i_\Theta(\sigma_i(\bs{\mathfrak{s}})) \in \ol{Q}(\wt{P}_\Phi^{i+1}(\sigma_{i+1}(\bs{\mathfrak{s}})), \ell^{i+1}_\Phi),
    \label{eqn: in the box hypo}
\end{eqnarray}
after using the fact that
\begin{eqnarray}
    \wt{P}_\Phi^{i+1}(\sigma_{i+1}(\bs{\mathfrak{s}})) = P_\Phi^{i+1}(\sigma_{i+1}(\bs{\mathfrak{s}})) - P^i_\Phi(\sigma_i(\bs{\mathfrak{s}})) + P^i_\Theta(\sigma_i(\bs{\mathfrak{s}})).
    \label{eqn: center identity repeat}
\end{eqnarray}
Using (\ref{eqn: in the box hypo}), property (viii) and definition (\ref{eqn: traj gamma v}), we have that
\begin{align}
    \gamma^{\bs{v}}_{\bs{x}_s}(\tau_{i+1}) & = \mathcal{P}_{\Phi}(\bs{\mathfrak{s}}) - P^i_\Phi(\sigma_i(\bs{\mathfrak{s}})) + P^i_\Theta(\sigma_i(\bs{\mathfrak{s}})) - \wt{P}_\Phi^{i+1}(\sigma_{i+1}(\bs{\mathfrak{s}})) + P_\Theta^{i+1}(\sigma_{i+1}(\bs{\mathfrak{s}})) \nonumber \\
    & = \mathcal{P}_{\Phi}(\bs{\mathfrak{s}}) - P_\Phi^{i+1}(\sigma_{i+1}(\bs{\mathfrak{s}})) + P_\Theta^{i+1}(\sigma_{i+1}(\bs{\mathfrak{s}})), \nonumber
\end{align}
where we used (\ref{eqn: center identity repeat}) to obtain the second line. 
Finally, taking the limit $i \to \infty$, we see that
\begin{align}
    \gamma^{\bs{v}}_{\bs{x}_s}(\tau_\infty) = \mathcal{P}_\Theta(\bs{\mathfrak{s}}). \nonumber
\end{align}
\end{proof}
\subsection{Modifications required to prove Theorem \ref{thm steady}}
Our design of a steady vector field $\bs{u}^s:\mathbb{T}^d \to \mathbb{R}^d$ for $d \geq 3$ is based on the unsteady vector field construction of dimension $d-1$, which we denote as $\bs{u}^{d-1}$ in this subsection. Let $\bs{x} = (x_1, x_2, \dots x_d) \in \mathbb{T}^d$, we denote $\bs{x}^\prime = (x_1, \dots x_{d-1}) \in \mathbb{T}^{d-1}$, and we write $\bs{x} = (\bs{x}^\prime, x_d)$. In our definition below, the coordinate $x_d$ serves as a function of time. Given $0 < \varepsilon < 1$, we define
\begin{align}
    \bs{u}^s(\bs{x}^\prime, x_d) \coloneqq 
    \begin{cases}
    (0, \dots, 0, 1) \quad & \text{if} \quad  0 \leq x_d < 1 - \varepsilon, \\
    \left(\frac{1}{\varepsilon}\bs{u}^{d-1}\left(\frac{x_d - (1 - \varepsilon)}{\varepsilon}, \bs{x}^\prime\right), 1\right) \quad & \text{if} \quad  1 - \varepsilon \leq x_d \leq 1,
    \end{cases}
\end{align}
where we use final time $T = 1$ in the definition of $\bs{u}^{d-1}$.
Using the properties of the unsteady vector field $\bs{u}^{d-1}$, it is not difficult to show that $\bs{u}^s$ is indeed the required vector field in Theorem \ref{thm steady}.

\section{Blob flow}
\label{sec: blob flow}
As described in the approach, to translate a cube in the domain from a starting point $\bs{x}_s$ to an endpoint $\bs{x}_e$, we use what we call a ``blob flow.'' A schematic of a blob flow is shown in figure \ref{fig: blob}. The properties of a blob flow $\wt{\bs{v}}$ are specified in Proposition \ref{Prop of v tilde}. For a blob flow, the vector field $\wt{\bs{v}}$ is uniform inside a cube of length $\lambda$, and the support of the vector field lies in a cube of a slightly larger length $\lambda (1 + \delta)$, where $\delta$ can be understood as an offset. The vector field folds back outside the cube of length $\lambda$ to maintain the divergence-free condition  (see figure \ref{fig: blob}). Our goal now is to first construct a stationary blob flow, using which we construct the required blob flow and finally give proof of Proposition \ref{Prop of vi}.
\subsection{A stationary blob flow}
Let us first define a bump function as
\begin{eqnarray}
    \varphi(x) \coloneqq 
    \begin{cases}
        c \exp\left(\frac{1}{x^2-\frac{1}{4}}\right)  \quad & \text{if} \quad x \in \left(-\frac{1}{2}, \frac{1}{2}\right), \\
        0 \quad & \text{if} \quad x \in \left(-\infty, -\frac{1}{2}\right] \cup \left[\frac{1}{2}, \infty\right),
    \end{cases}
    \label{def: varphi}
\end{eqnarray}
where the constant $c$ is chosen such that $\int_{\mathbb{R}} \varphi(x^\prime) \, {\rm d} x^\prime = 1$. For any $\varepsilon > 0$, we define a standard mollifier as
\begin{eqnarray}
    \varphi_\varepsilon(x) \coloneqq \frac{1}{\varepsilon} \varphi\left(\frac{x}{\varepsilon}\right).
\end{eqnarray}
Next, for $\delta \in (0, 1)$, let $\chi_{[-\frac{1}{2} -\frac{3 \delta}{8}, \frac{1}{2} + \frac{3 \delta}{8}]}$ be an indicator function which is $1$ if $x \in [-\frac{1}{2} -\frac{3 \delta}{8}, \frac{1}{2} + \frac{3 \delta}{8}]$ and zero otherwise. We define a function $\zeta_1 : \mathbb{R} \to \mathbb{R}$ as 
\begin{eqnarray}
    \zeta_1 \coloneqq \varphi_{\frac{\delta}{8}} \ast \chi_{[-\frac{1}{2} -\frac{3 \delta}{8}, \frac{1}{2} + \frac{3 \delta}{8}]}. 
\end{eqnarray}
For $d \geq 2$, we define $\zeta_d : \mathbb{R}^d \to \mathbb{R}$ as
\begin{eqnarray}
    \zeta_d(\bs{x}) \coloneqq \zeta_1(x_1) \zeta_1(x_2) \dots \zeta_1(x_d). 
\end{eqnarray}
It is a standard calculation to check that $\zeta_d \in C^\infty_c(\mathbb{R}^d)$, $\supp \zeta_d \subseteq [-\frac{1}{2} -\frac{\delta}{2}, \frac{1}{2} + \frac{\delta}{2}]^d$, $\norm{\zeta_d}_{L^\infty} = 1$ and that $\zeta_d(\bs{x}) = 1$ if $\bs{x} \in [-\frac{1}{2}, \frac{1}{2}]^d$. Furthermore, $\norm{\nabla^i \zeta_d}_{L^\infty} \lesssim \frac{1}{\delta^{i}}$ for $i \in \mathbb{N}$.

Let $\bs{q} \in \mathbb{S}^{d - 1}$ and $d = 2k$ or $2k+1$ for some $k \in \mathbb{N}$, we define a function $F_1 : \mathbb{R}^d \to \mathbb{R}$ as
\begin{eqnarray}
    F_1(\bs{x}) \coloneqq \left(q_1 x_2 - q_2 x_1 + q_3 x_4 - q_4 x_3 \dots + q_{2k-1} x_{2k} - q_{2k} x_{2k-1}\right) \zeta_d(\bs{x}).
\end{eqnarray}
When $d = 2k+1$, we also define a function $F_2 : \mathbb{R}^d \to \mathbb{R}$ as
\begin{eqnarray}
    F_2(\bs{x}) \coloneqq q_{2k+1} x_1 \zeta_d(\bs{x}).
\end{eqnarray}
Finally, we define the stationary blob flow $\bs{w}:\mathbb{R}^d \to \mathbb{R}^d$. When $d = 2k$, we let
\begin{eqnarray}
    \bs{w} = (w_1, w_2, \dots, w_{2k}) \coloneqq \left(\partial_{x_2} F_1, - \partial_{x_1} F_1, \partial_{x_4} F_1, -\partial_{x_3} F_1, \dots, \partial_{x_{2k}} F_1, -\partial_{x_{2k-1}} F_1\right),
\end{eqnarray}
and when $d = 2k+1$, we let
\begin{eqnarray}
    \bs{w} = (w_1, w_2, \dots, w_{2k+1}) \coloneqq \left(\partial_{x_2} F_1 - \partial_{x_{2k+1}} F_2, - \partial_{x_1} F_1, \partial_{x_4} F_1, -\partial_{x_3} F_1, \dots, \partial_{x_{2k}} F_1, -\partial_{x_{2k-1}} F_1, \partial_{x_1} F_2    \right).
\end{eqnarray}
From the definition itself, it is clear that $\bs{w}$ is divergence-free. In general, we record the properties of $\bs{w}$ in the following lemma.
\begin{lemma}
\label{Prop of w}
Depending on two parameters, $\bs{q} \in \mathbb{S}^{d-1}$ and $\delta \in (0, 1)$, there is a divergence-free vector field $\bs{w}(\cdot \, ; \bs{q}, \delta) \in C_c^\infty(\mathbb{R}^d, \mathbb{R}^d)$ with the following properties. \\[5pt]
(i) $\supp \bs{w} \subseteq \ol{Q}(\bs{0}, 1 + \delta)$, \\[5pt]
(ii) If $\bs{x} \in \ol{Q}(\bs{0}, 1)$, then $\bs{w}(\bs{x}) = \bs{q}$, \\[5pt]
(iii) $\norm{\bs{w}}_{L^\infty} \lesssim \frac{1}{\delta}$, \\[5pt]
(iv) $\norm{\nabla^i \bs{w}}_{L^\infty} \lesssim \frac{1}{\delta^{i+1}} \quad \forall \; i \in \mathbb{N}$.
\end{lemma}

 \begin{figure}[h]
\centering
 \includegraphics[scale = 0.45]{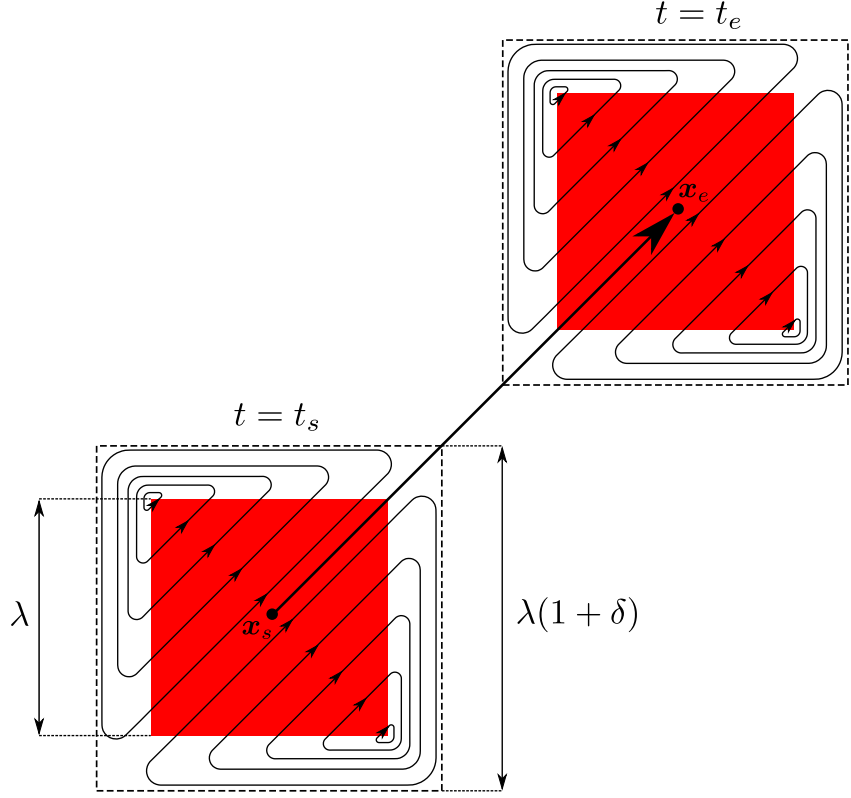}
 \caption{An illustration of the blob flow $\wt{\bs{v}}$ in two dimensions. The vector field $\wt{\bs{v}}$ translates a cube of length $\lambda$ centered at $\bs{x}_s$ at time $t = t_s$ to a cube centered at $\bs{x}_e$ at $t = t_e$. The cube is shown in red color. As the cube moves, the vector field $\wt{\bs{v}}$ inside the red cube is always uniform, and the support of the vector field lies in a slightly bigger cube of length $\lambda(1 + \delta)$ (shown in dashed).}
 \label{fig: blob}
\end{figure}
\subsection{A moving blob}
Using the vector field $\bs{w}$, our goal now is to design an unsteady vector field $\wt{\bs{v}}$ which translates a cube centered around a starting point $\bs{x}_s$ to an endpoint $\bs{x}_e$. We first define $\eta : \mathbb{R} \to \mathbb{R}$ as  
\begin{eqnarray}
    \eta(x) \coloneqq \int_{-\infty}^{x-\frac{1}{2}} \varphi(x^\prime) \, {\rm d} x^\prime,
\end{eqnarray}
where $\varphi$ is given by (\ref{def: varphi}). It is clear that $\eta(x) = 0$ if $x \leq 0$ and $\eta(x) = 1$ if $x \geq 1$.
Given two points $\bs{x}_s, \bs{x}_e \in \mathbb{R}^d$ and time $t_s < t_e$, we define $\bs{x}_p: \mathbb{R} \to \mathbb{R}^d$ as
\begin{eqnarray}
    \bs{x}_p(t) \coloneqq \bs{x}_s \left[1 - \eta\left(\frac{t - t_s}{t_e - t_s}\right)\right]  + \bs{x}_e \eta\left(\frac{t - t_s}{t_e - t_s}\right). 
\end{eqnarray}
Finally, given $0 < \delta < 1$ and $\lambda > 0$, we define a vector field $\wt{\bs{v}}(\cdot \, ; \bs{x}_s, \bs{x}_e, t_s, t_e, \lambda, \delta) : \mathbb{R} \times \mathbb{R}^d \to \mathbb{R}^d$ as
\begin{eqnarray}
    \wt{\bs{v}}(t, \bs{x}) \coloneqq |\bs{x}_p^\prime(t)| \bs{w}\left(\frac{\bs{x} - \bs{x}_p(t)}{\lambda} ; \frac{\bs{x}_e - \bs{x}_s}{|\bs{x}_e - \bs{x}_s|}, \delta \right).
    \label{eqn: v tilde}
\end{eqnarray}
Here, $\bs{x}_p(t)$ signifies the trajectory of the center of the cube that $\wt{\bs{v}}$ would translate. We collect all the important properties of the vector field $\wt{\bs{v}}$ in the following proposition.
\begin{proposition}
\label{Prop of v tilde}
Given a few parameters $\bs{q} \in \mathbb{S}^{d-1}$, $\delta \in (0, 1)$, $\bs{x}_s, \bs{x}_e \in \mathbb{R}^d$ and  $0 \leq t_s < t_e$, then the vector field $\wt{\bs{v}}(\cdot \, ; \bs{x}_s, \bs{x}_e, t_s, t_e, \lambda, \delta) : \mathbb{R} \times \mathbb{R}^d \to \mathbb{R}^d$ as defined in (\ref{eqn: v tilde}) has the following properties. \\[5pt]
(i) $\wt{\bs{v}} \in C_c^\infty(\mathbb{R} \times \mathbb{R}^d, \mathbb{R}^d)$, \\[5pt]
(ii) $\wt{\bs{v}}$ is divergence-free, \\[5pt]
(iii) $\supp_t \wt{\bs{v}} \subseteq [t_s, t_e]$, \\[5pt]
(iv) $\supp \wt{\bs{v}}(t, \cdot) \subseteq \ol{Q}(\bs{x}_p(t), \lambda(1 + \delta))$, \\[5pt]
(v) If $\bs{x} \in \ol{Q}(\bs{x}_p(t), \lambda)$ then $\wt{\bs{v}}(t, \bs{x}) = \frac{(\bs{x}_e - \bs{x}_s)}{t_e - t_s} \eta^\prime\left(\frac{t - t_s}{t_e - t_s}\right)$, \\[5pt]
(vi) $\norm{\wt{\bs{v}}}_{L^\infty_{t, \bs{x}}} \lesssim \frac{|\bs{x}_e - \bs{x}_s|}{\delta (t_e - t_s)}$, \\[5pt]
(vii) $ \norm{\nabla^i \wt{\bs{v}}}_{L^\infty_{t, \bs{x}}} \lesssim \frac{1}{\lambda^i \delta^{i+1}} \frac{|\bs{x}_e - \bs{x}_s|}{t_e - t_s} \quad \forall i \in \mathbb{N}$, \\[5pt]
(viii) $\norm{\partial_t \wt{\bs{v}}}_{L^\infty_{t, \bs{x}}} \lesssim \frac{|\bs{x}_e - \bs{x}_s|}{\delta (t_e - t_s)^2} + \frac{1}{\lambda \delta^2}\frac{|\bs{x}_e - \bs{x}_s|^2}{(t_e - t_s)^2}$, \\[5pt]
(ix) If $\bs{x} \in \ol{Q}(\bs{x}_s, \lambda)$ then $\gamma^{\wt{\bs{v}}}_{\bs{x}}(t_e) = \bs{x} + \bs{x}_e - \bs{x}_s$.
\end{proposition}
\begin{proof}[Proof of Proposition \ref{Prop of v tilde}]
The proof of all these properties is straightforward. Properties (i) to (v) is straight from the definition and Lemma \ref{Prop of w}.We perform a few simple computations to show properties (vi) to (viii). For the property (vi), we have
\begin{align}
& \wt{\bs{v}} = \frac{|\bs{x}_e - \bs{x}_s|}{t_e - t_s} \eta^\prime\left(\frac{t - t_s}{t_e - t_s}\right) \bs{w}\left(\frac{\bs{x} - \bs{x}_p(t)}{\lambda} ; \frac{\bs{x}_e - \bs{x}_s}{|\bs{x}_e - \bs{x}_s|}, \delta \right) \nonumber \\
\implies &  \norm{\wt{\bs{v}}}_{L^\infty_{t, \bs{x}}} \leq \frac{|\bs{x}_e - \bs{x}_s|}{t_e - t_s} \sup_{t} |\eta^\prime| \norm{\bs{w}}_{L^\infty} \lesssim \frac{|\bs{x}_e - \bs{x}_s|}{\delta (t_e - t_s)}. \nonumber
\end{align}
The property (vii) can be shown to hold as follows.
\begin{align}
& \nabla^i \wt{\bs{v}} = \frac{1}{\lambda^i} \frac{|\bs{x}_e - \bs{x}_s|}{t_e - t_s} \eta^\prime\left(\frac{t - t_s}{t_e - t_s}\right) (\nabla^i \bs{w})\left(\frac{\bs{x} - \bs{x}_p(t)}{\lambda} ; \frac{\bs{x}_e - \bs{x}_s}{|\bs{x}_e - \bs{x}_s|}, \delta \right) \nonumber \\
\implies &  \norm{\nabla^i \wt{\bs{v}}}_{L^\infty_{t, \bs{x}}} \leq \frac{1}{\lambda^i}\frac{|\bs{x}_e - \bs{x}_s|}{t_e - t_s} \sup_{t} |\eta^\prime| \norm{\nabla^i \bs{w}}_{L^\infty} \lesssim \frac{1}{\lambda^i \delta^{i+1}}\frac{|\bs{x}_e - \bs{x}_s|}{t_e - t_s}. \nonumber
\end{align}
The property (viii) is also proved through a simple computation.
\begin{align}
\partial_t \wt{\bs{v}} = & \frac{|\bs{x}_e - \bs{x}_s|}{(t_e - t_s)^2} \eta^{\prime \prime} \left(\frac{t - t_s}{t_e - t_s}\right) \bs{w}\left(\frac{\bs{x} - \bs{x}_p(t)}{\lambda} ; \frac{\bs{x}_e - \bs{x}_s}{|\bs{x}_e - \bs{x}_s|}, \delta \right) \nonumber \\
& + \frac{|\bs{x}_e - \bs{x}_s|}{t_e - t_s} \eta^\prime\left(\frac{t - t_s}{t_e - t_s}\right) \left[(\nabla \bs{w})\left(\frac{\bs{x} - \bs{x}_p(t)}{\lambda} ; \frac{\bs{x}_e - \bs{x}_s}{|\bs{x}_e - \bs{x}_s|}, \delta \right) \cdot \left(\frac{(\bs{x_s} - \bs{x}_e)}{\lambda (t_e - t_s)} \eta^\prime\left(\frac{t - t_s}{t_e - t_s}\right) \right)\right] \nonumber \\
\implies \norm{\partial_t \wt{\bs{v}}}_{L^\infty_{t, \bs{x}}} & \lesssim \frac{|\bs{x}_e - \bs{x}_s|}{\delta(t_e - t_s)^2} + \frac{1}{\lambda \delta^2}\frac{|\bs{x}_e - \bs{x}_s|^2}{(t_e - t_s)^2} \nonumber
\end{align}
Finally, for the property (ix), we claim that 
\begin{align}
  \gamma^{\wt{\bs{v}}}_{\bs{x}}(t) \coloneqq \bs{x} + (\bs{x}_e - \bs{x}_s) \eta \left(\frac{t - t_s}{t_e - t_s}\right) \qquad \text{for} \quad   \bs{x} \in \ol{Q}(\bs{x}_s, \lambda) \nonumber
\end{align}
is a trajectory starting from $\bs{x}$. It is easy to see that in the above definition $\gamma^{\wt{\bs{v}}}_{\bs{x}}(t) \in \ol{Q}(\bs{x}_p(t), \lambda)$, therefore, from the property (v), we have
\begin{align}
    \wt{\bs{v}}(t, \gamma^{\wt{\bs{v}}}_{\bs{x}}(t)) = \frac{(\bs{x}_e - \bs{x}_s)}{t_e - t_s} \eta^\prime\left(\frac{t - t_s}{t_e - t_s}\right). \nonumber
\end{align}
Next, we simply verify that 
\begin{eqnarray}
    \frac{d \gamma^{\wt{\bs{v}}}_{\bs{x}}(t)}{d t} = \wt{\bs{v}}(t, \gamma^{\wt{\bs{v}}}_{\bs{x}}(t)) \qquad \text{and} \qquad \gamma^{\wt{\bs{v}}}_{\bs{x}}(0) = \bs{x}. \nonumber
\end{eqnarray}
Finally, noting that $\gamma^{\wt{\bs{v}}}_{\bs{x}}(t_e) = \bs{x} + \bs{x}_e - \bs{x}_s$ completes the proof of the proposition.
\end{proof}

\subsection{Assembly of moving blobs: A proof of Proposition \ref{Prop of vi}}
\label{subsection: assembly of moving blobs}
\begin{proof}[Proof of Proposition \ref{Prop of vi}]
Given $\mathfrak{s} \in S_{i+1}$ for $i \in \mathbb{Z}_{\geq 0}$, let us first define $\ol{\bs{v}}_\mathfrak{s}:[0, \tau_\infty] \times \mathbb{T}^d \to \mathbb{R}^d$ as
\begin{eqnarray}
    \ol{\bs{v}}_{\mathfrak{s}}(t, \bs{x}) \coloneqq \wt{\bs{v}}\left(t, \bs{x}; \wt{P}^{i+1}_{\Phi}(\mathfrak{s}), P^{i+1}_{\Theta}(\mathfrak{s}), \frac{2 \tau_i + \tau_{i+1}}{3}, \frac{\tau_i + 2\tau_{i+1}}{3}, \ell^{i+1}_\Phi, \vartheta(\nu)\right),
\end{eqnarray}
where 
\begin{eqnarray}
    \vartheta(\nu) = \frac{2^\nu - 1}{8}, \nonumber
\end{eqnarray}
as defined in (\ref{def: varepsilon}).
In the above defintion, $\wt{\bs{v}}:[0, \tau_\infty] \times \mathbb{T}^d \to \mathbb{R}^d$ is the $\mathbb{T}^d$-periodized version of the $\wt{\bs{v}}$ from Proposition \ref{Prop of v tilde} restricted to time $0$ to $\tau_\infty$.
Next, we define $\bs{v}_i: [0, \tau_\infty] \times \mathbb{T}^d \to \mathbb{R}^d$ as
\begin{eqnarray}
    \bs{v}_i \coloneqq \sum_{\mathfrak{s} \in S_{i+1}} \ol{\bs{v}}_{\mathfrak{s}}.
    \label{def: definition vi}
\end{eqnarray}
We claim that $\bs{v}_i$ satisfies all the properties as specified in Proposition \ref{Prop of vi}. Noting properties (i) and (ii) from Proposition \ref{Prop of v tilde}, it is clear that $\bs{v}_i \in C^\infty([0, \tau_\infty] \times \mathbb{T}^d, \mathbb{R}^d)$ and that $\bs{v}_i$ is divergence-free. Further, using property (iii) from Proposition \ref{Prop of v tilde}, we see that
\begin{eqnarray}
    \supp_t \bs{v}_i \subseteq \left[\frac{2 \tau_i + \tau_{i+1}}{3}, \frac{\tau_i + 2 \tau_{i+1}}{3}\right] \Subset (\tau_i, \tau_{i+1}). \nonumber
\end{eqnarray}
Before proving the rest of the properties in Proposition \ref{Prop of vi}, we first notice that in the definition (\ref{def: definition vi}), $\supp \ol{\bs{v}}_{\mathfrak{s}}(t, \cdot)$ are disjoint. Using the property (iv) in Proposition \ref{Prop of v tilde}, for any $\mathfrak{s} \in S_{i+1}$, we note that
\begin{eqnarray}
   \supp \ol{\bs{v}}_{\mathfrak{s}}(t, \cdot) \subseteq \ol{Q}(\bs{x}_{\mathfrak{s}}(t), \ell^{i+1}_\Phi (1 + \vartheta)), \nonumber 
\end{eqnarray}
where 
\begin{eqnarray}
    \bs{x}_{\mathfrak{s}}(t) = \wt{P}^{i+1}_{\Phi}(\mathfrak{s}) \left[1 - \eta\left(\frac{3t - 2 \tau_i - \tau_{i+1}}{\tau_{i+1} - \tau_i}\right)\right]  +  P^{i+1}_{\Theta}(\mathfrak{s})\eta\left(\frac{3t - 2 \tau_i - \tau_{i+1}}{\tau_{i+1} - \tau_i}\right). \nonumber
\end{eqnarray}
Next, using Lemma (\ref{lemma: an important subset}), we have that 
\begin{eqnarray}
    \supp \ol{\bs{v}}_{\mathfrak{s}}(t, \cdot) \subseteq \ol{Q}(\bs{x}_{\mathfrak{s}}(t), \ell^{i+1}_\Phi (1 + \vartheta)) \Subset Q(P^{i+1}_\Theta(\mathfrak{s}), \ell^{i+1}_\Theta). \nonumber
\end{eqnarray}
From Lemma \ref{lemma: disjoint open cubes}, the open cubes $Q(P^{i+1}_\Theta(\mathfrak{s}), \ell^{i+1}_\Theta)$ are disjoint, which in turn implies that $\supp \ol{\bs{v}}_{\mathfrak{s}}(t, \cdot)$ are disjoint, an important detail we  frequently use in the rest of the proof. Using property (vi) from Proposition \ref{Prop of v tilde}, we obtain
\begin{align}
    \norm{\bs{v}_i}_{L^\infty_{t, \bs{x}}} \leq \max_{\mathfrak{s} \in S_{i+1}} \norm{\ol{\bs{v}}_{\mathfrak{s}}}_{L^\infty_{t, \bs{x}}} \lesssim \frac{1}{(2^\nu - 1)} \frac{\ell^{i+1}_\Theta}{(\tau_{i+1} - \tau_{i})} \leq \frac{1}{(2^{\nu}-1)} \frac{1}{2^{i \beta}}. \nonumber
\end{align}
Using property (vii) from Propostion \ref{Prop of v tilde}, we have
\begin{align}
    \norm{\nabla \bs{v}_i}_{L^\infty_{t, \bs{x}}} \leq \max_{\mathfrak{s} \in S_{i+1}} \norm{\nabla \ol{\bs{v}}_{\mathfrak{s}}}_{L^\infty_{t, \bs{x}}} \lesssim \frac{1}{\ell^{i+1}_\Phi (2^\nu - 1)^2} \frac{\ell^{i+1}_\Theta}{(\tau_{i+1} - \tau_{i})} \leq \frac{2^{(1 + \nu - \beta) i}}{(2^{\nu}-1)^2}. \nonumber
\end{align}
The Sobolev norm of the vector field $\bs{v}_i$ can be obtained after using properties (vi) and (vii) as
\begin{align}
    \norm{ \bs{v}_i(t, \cdot)}_{W^{1, p}} & \leq \left(\sum_{\mathfrak{s} \in S_{i+1}} \left(\norm{\ol{\bs{v}}_{\mathfrak{s}}}_{L^\infty_{t, \bs{x}}}^p + \norm{\nabla \ol{\bs{v}}_{\mathfrak{s}}}_{L^\infty_{t, \bs{x}}}^p\right)\mathscr{L}^d(\supp \ol{\bs{v}}_{\mathfrak{s}}(t, \cdot))\right)^{\frac{1}{p}} \nonumber \\
    & \lesssim \left(\sum_{\mathfrak{s} \in S_{i+1}} \left(\frac{2^{(1 + \nu - \beta)pi}}{(2^{\nu} - 1)^{2p}} \; \frac{1}{2^{(1+\nu)d(i+1)}}\right)\right)^{\frac{1}{p}} \nonumber \\
    & \lesssim \frac{1}{(2^\nu - 1)^2} \times 2^{\frac{[(1 + \nu - \beta)p - d \nu]}{p}i}. \nonumber
\end{align}
An upper bound on the time derivative of the vector field $\bs{v}_i$ can be obtained after using (viii) as
\begin{align}
    \norm{\partial_t \bs{v}_i}_{L^\infty_{t, \bs{x}}} \leq \max_{\mathfrak{s} \in S_{i+1}} \norm{\partial_t \ol{\bs{v}}_{\mathfrak{s}}}_{L^\infty_{t, \bs{x}}} & \lesssim \frac{1}{(2^\nu - 1)} \frac{\ell^{i+1}_\Theta}{(\tau_{i+1} - \tau_{i})^2} + \frac{1}{\ell^{i+1}_{\Phi}(2^\nu - 1)^2} \frac{(\ell^{i+1}_\Theta)^2}{(\tau_{i+1} - \tau_{i})^2} \nonumber \\
    & \lesssim \frac{1}{(2^\nu - 1)^2}  \frac{1}{2^{[2 \beta - \nu - 1]i}}. \nonumber
\end{align}
Finally, using the fact that at any point in time $ \supp \ol{\bs{v}}_{\mathfrak{s}}(t, \cdot) \Subset Q(P^{i+1}_\Theta(\mathfrak{s}), \ell^{i+1}_\Theta)$ and the property (ix) in Proposition \ref{Prop of v tilde}, we conclude that if for some $\mathfrak{s} \in S_{i+1}$, $\bs{x} \in \ol{Q}(\wt{P}^{i+1}_\Phi(\mathfrak{s}), \ell^{i+1}_\Phi)$, then $\gamma^{\bs{v}_i}_{\bs{x}}(\tau_{i+1}) = \bs{x} - \wt{P}^{i+1}_\Phi(\mathfrak{s}) + P^{i+1}_\Theta(\mathfrak{s}).$
\end{proof}

\bibliographystyle{halpha-abbrv}
\bibliography{references.bib}

\end{document}